\documentclass[a4paper, 10pt, notitlepage]{article}

\usepackage{amsthm, amsmath, amssymb, latexsym}
\usepackage{mathrsfs,cite}
\usepackage[multiple]{footmisc}
\usepackage{graphicx}
\usepackage[ruled,vlined]{algorithm2e}

\theoremstyle{plain}
\newtheorem{theorem}{Theorem}
\newtheorem{lemma}{Lemma}

\newtheorem{proposition}{Proposition}

\usepackage[a4paper]{geometry}		% Changing margins.
 
\theoremstyle{definition}

\theoremstyle{remark}
\newtheorem{remark}{Remark}

\DeclareMathOperator{\co}{co}
\DeclareMathOperator{\bd}{bd}
\DeclareMathOperator*{\argmin}{arg\,min}
\DeclareMathOperator{\proj}{Pr}
\DeclareMathOperator{\determ}{det}

\author{M.V. Dolgopolik\footnote{Institute for Problems in Mechanical Engineering, Russian Academy of Sciences,
Saint Petersburg, Russia}
\footnote{This work was performed in IPME RAS and supported by the Russian Science Foundation (Grant No. 20-71-10032).}}
\title{The alternating direction method of multipliers for finding the distance between ellipsoids}

\begin{document}

\maketitle

\begin{abstract}
We study several versions of the alternating direction method of multipliers (ADMM) for solving the convex problem of
finding the distance between two ellipsoids and the nonconvex problem of finding the distance between the boundaries of
two ellipsoids. In the convex case we present the ADMM with and without automatic penalty updates and demonstrate via
numerical experiments on problems of various dimensions that our methods significantly outperform all other existing
methods for finding the distance between ellipsoids. In the nonconvex case we propose a heuristic rule for updating the
penalty parameter and a heuristic restarting procedure (a heuristic choice of a new starting for point for the second
run of the algorithm). The restarting procedure was verified numerically with the use of a global method based on KKT
optimality conditions. The results of numerical experiments on various test problems showed that this procedure always
allows one to find a globally optimal solution in the nonconvex case. Furthermore, the numerical experiments also
demonstrated that our version of the ADMM significantly outperforms existing methods for finding the distance between
the boundaries of ellipsoids on problems of moderate and high dimensions.
\end{abstract}

\section{Introduction}

The alternating direction method of multipliers (ADMM) is an efficient method for solving structured convex optimisation
problems \cite{EcksteinYao,YangHan2016,HanSunZhang2018,HeYangWang,GhadimiTeixeiraShames,BotCsetnek} that has found a
wide variety of applications \cite{BoydParikhChuEtAl,JiaCaiHan,FangHeLiuYuan2015,LingRibeiro2014}, including
applications to some nonconvex and nonsmooth optimisation problems \cite{YangPongChen2017,HajinezhadShi}. Although the
ADMM was originally developed for structured convex problems, extensions of this method to various nonconvex settings
have recently become an active area of research
\cite{HongLuoRazaviyayn,BotNguyen,GuoHanWu2017,MagnussonRabbat2016,ZhangLuo2020,PengChenZhu2015,WangYinZeng,
ThemelisPatrinos}.

Recently, a version of the ADMM for finding the Euclidean projection of a point onto an ellipsoid has been developed
\cite{JiaCaiHan}. Numerical experiments presented in \cite{JiaCaiHan} showed that the ADMM is significantly faster
than other existing methods for projecting a point onto an ellipsoid.

Being inspired by paper \cite{JiaCaiHan}, we propose to apply the ADMM to the problem of finding the distance between
two ellipsoids or the boundaries of two ellipsoids, which have attracted a considerable attention of researchers. A
detailed algebraic analysis of these and related problems was presented in
\cite{UteshevGoncharova2018,UteshevYashina2015,UteshevYashina2008}. An exact penalty method for finding the distance
between the boundaries of two ellipsoids was developed in \cite{TamasyanChumakov}. A geometric method for finding the
distance between two ellipsoids was studied in \cite{LinHan}, while the so-called charged balls method for solving this
problem was proposed in \cite{Abbasov}. A global optimisation method for solving a closely related (but more difficult)
nonconvex problem of finding the so-called \textit{signed} distance between two ellipsoids was developed in 
\cite{IwataNakatsukasaTakeda}.

The main goal of this paper is to develop the ADMM for solving both the convex problem of finding the distance between
two ellipsoids and the nonconvex problem of finding the distance between the boundaries of two ellipsoids. In the convex
case we present the standard version of the ADMM with fixed penalty parameter and, following the ideas of
\cite{HeYangWang}, a version of the ADMM with automatic adjustments of the penalty parameter. To verify the efficiency
of the proposed method, we present some results of numerical experiments on problems of various dimensions. These
results demonstrate that the ADMM with penalty adjustments is faster than the ADMM with fixed penalty parameter, and
both versions of the ADMM significantly outperform all other existing methods for finding the distance between
ellipsoids.

In the nonconvex case, we present a version of the ADMM with a heuristic rule for updating the penalty parameter. We
provide a theoretical analysis of this method with a different rule for updating the penalty parameter. This rule
performed poorly in numerical experiments; nevertheless, its analysis sheds some light on the overall performance of the
ADMM and helps one to gain an insight into the choice of parameters of the method and their role in the optimisation
process.

Since the problem of finding the distance between the boundaries of two ellipsoids is essentially nonconvex (i.e. there
is always a locally optimal solution of this problem, which is not globally optimal), we also propose a heuristic
restarting procedure for our method. Namely, we propose a heuristic choice of a new starting point for the second run
of the algorithm that is defined by a point computed during the first run. To verify whether this strategy allows one
to find a globally optimal solution, we describe a slight modification of the method from
\cite{IwataNakatsukasaTakeda} that provably computes a globally optimal solution of the problem of finding the distance
between the boundaries of two ellipsoids. Then we present and discuss some results of numerical experiments on problems
of various dimensions. These results (as well as results of numerous other experiments on various test problems not
presented in this paper) demonstrate that the proposed heuristic restarting procedure always allows one to find a 
globally optimal solution. Furthermore, we applied our restarting procedure to the exact penalty method from
\cite{TamasyanChumakov}, and the results of numerical experiments showed that this method with the restarting procedure
always finds a globally optimal solution as well, but without this procedure the method often converges to a locally
optimal solution. Finally, the numerical experiments show that the ADMM significantly outperforms other methods for
finding the distance between the boundaries of ellipsoids on problems of moderate and high dimension, while on small
dimensional problems (namely, when the dimension is smaller than $8$), the modification of the global method from
\cite{IwataNakatsukasaTakeda} is the fastest method.

The paper is organised as follows. The convex problem of finding the distance between two ellipsoids is studied in 
Section~\ref{sect:ConvexCase}. In Subsection~\ref{subsect:ADMM_ConvexCase} we present two versions of the ADMM for
solving this problem, Subsection~\ref{subsect:ConvergenceConvexCase} contains a theoretical result on 
the convergence of these methods, while the results of numerical experiments are given in 
Subsection~\ref{subsect:NumericalExperiments_ConvexCase}. Section~\ref{sect:NonconvexCase} is devoted to the nonconvex
problem of finding the distance between the boundaries of two ellipsoids. The ADMM for solving this problem is
described in Subsection~\ref{subsect:ADMM_Nonconvex}. Its theoretical analysis is presented in
Subsection~\ref{subsect:Theoretical_Analysis}, while a heuristic restarting procedure for the method and a global method
for solving the problem are considered in Subsection~\ref{sect:HeuristicRestart}. Finally, some results of numerical
experiments in the nonconvex case are given in Subsection~\ref{subsect:NumericalExperiments_NonconvexCase}.

\section{The distance between ellipsoids: the convex case}
\label{sect:ConvexCase}

In the first part of the paper we study the problem of finding the distance between two ellipsoids $\mathcal{E}_1$ and
$\mathcal{E}_2$ in $\mathbb{R}^d$, defined as
\begin{equation} \label{eq:EllipsoidsDef}
  \mathcal{E}_i = \Big\{ x \in \mathbb{R}^d \Bigm| \big\langle x - z_i, Q_i (x - z_i) \big\rangle \le 1 \Big\}, 
  \quad i \in \{ 1, 2 \}.
\end{equation}
Here $\langle \cdot, \cdot \rangle$ is the inner product in $\mathbb{R}^d$, $z_1, z_2 \in \mathbb{R}^d$ are the
centres of the ellipsoids, and $Q_1, Q_2 \in \mathbb{R}^{d \times d}$ are positive definite symmetric matrices. The
problem of finding the distance between these ellipsoids is a convex programming problem that can be formalised in the
following way:
\begin{equation} \label{prob:EllipsoidDist_Convex}
\begin{split}
  &\min\: \frac{1}{2} \| x_1 - x_2 \|^2 \\
  &\text{subject to} \quad
  \big\langle x_1 - z_1, Q_1 (x_1 - z_1) \big\rangle \le 1, \quad 
  \big\langle x_2 - z_2, Q_2 (x_2 - z_2) \big\rangle \le 1.
\end{split}
\end{equation}
Here $\| \cdot \|$ is the Euclidean norm and $x_1, x_2 \in \mathbb{R}^d$. Problem \eqref{prob:EllipsoidDist_Convex}
is a convex quadratically constrained quadratic programming problem, which means that almost all general convex
programming methods, such as interior point methods, can be used to find its solution. However, it is more efficient to
utilise a method exploiting the structure of problem \eqref{prob:EllipsoidDist_Convex}. Our main goal in this section is
to describe and analyse one such method.

\begin{remark}
Let us note that the main results of the paper can be easily extended to the case when the ellipsoids are defined as
$$
  \mathcal{E}_i = \Big\{ x \in \mathbb{R}^d \Bigm| 
  \big\langle x , A_i x \big\rangle + \big\langle b_i, x \big\rangle + \alpha_i \le 0 \Big\},
  \quad i \in \{ 1, 2 \},
$$
where the matrices $A_1$ and $A_2$ are positive definite and $b_1, b_2 \in \mathbb{R}^d$. Namely, if one puts
$$
  z_i = - \frac{1}{2} A^{-1}_i b_i, \quad Q_i = \frac{1}{- 0.5 \langle b_i, z_i \rangle - \alpha_i} A_i, 
  \quad i \in \{ 1, 2 \},
$$
then equalities \eqref{eq:EllipsoidsDef} hold true and one can apply the method presented in this article to the
problem under consideration. Note that $- 0.5 \langle b_i, z_i \rangle - \alpha_i > 0$, provided the ellipsoid 
$\mathcal{E}_i$ is nondegenerate (i.e. its interior is nonempty), since $z_i$ is the point of global minimum of the
quadratic function from the definition of $\mathcal{E}_i$, while its optimal value is equal to 
$0.5 \langle b_i, z_i \rangle + \alpha_i < 0$.
\end{remark}

\subsection{The alternating direction method of multipliers}
\label{subsect:ADMM_ConvexCase}

Being inspired by ideas of Jia, Cai, and Han \cite{JiaCaiHan} on algorithms for projecting a point onto an ellipsoid,
we propose to solve the problem of finding the distance between two ellipsoids with the use of the alternating direction
method of multipliers (ADMM). Recall that this method was originally developed to solve convex optimisation problems of
the form:
\begin{equation} \label{prob:SeparableConvexProblem}
  \min_{(x, y)} \: f(x) + g(y) \quad \text{subject to} \quad
  A x + B y = c, \quad x \in X, \quad y \in Y.
\end{equation}
Here $f \colon \mathbb{R}^d \to \mathbb{R}$ and $g \colon \mathbb{R}^m \to \mathbb{R}$ are convex functions, 
$A \in \mathbb{R}^{l \times d}$, $B \in \mathbb{R}^{l \times m}$, and $c \in \mathbb{R}^l$, while
$X \subseteq \mathbb{R}^d$ and $Y \subseteq \mathbb{R}^m$ are closed convex sets. A theoretical scheme of the ADMM for
solving problem \eqref{prob:SeparableConvexProblem} is as follows:
\begin{align*}
  x^{n + 1} &= \argmin_{x \in X} \left\{ f(x) - \langle \lambda^n, A x \rangle 
  + \frac{\tau}{2} \big\| A x + B y^n - c \big\|^2 \right\}, \\
  y^{n + 1} &= \argmin_{y \in Y} \left\{ g(y) - \langle \lambda^n, B y \rangle 
  + \frac{\tau}{2} \big\| A x^{n + 1} + B y - c \big\|^2 \right\}, \\
  \lambda^{n + 1} &= \lambda^n - \tau(A x^{n + 1} + B y^{n + 1} - c).
\end{align*}
Here $\lambda^n$ is an approximation of Lagrange multiplier for problem \eqref{prob:SeparableConvexProblem} and
$\tau > 0$ is the penalty parameter. It is well-known that, if well-defined, the ADMM converges to a globally optimal
solution of problem \eqref{prob:SeparableConvexProblem} for any value of the penalty
parameter $\tau$ under very mild assumptions (see \cite{BoydParikhChuEtAl,EcksteinYao,YangHan2016,HanSunZhang2018} for
more details). However, its numerical performance depends on the value of $\tau$. A poor choice of this parameter might
significantly slow down the convergence of the method. Therefore, below we will use a version of the ADMM with automatic
adjustments of the penalty parameter proposed in \cite{HeYangWang}, which will be described later.

In order to apply the ADMM to problem \eqref{prob:EllipsoidDist_Convex} let us us rewrite this problem as an
optimisation problem of the form \eqref{prob:SeparableConvexProblem}. Let $S_i$ be the square root of the matrix $Q_i$,
which exists since the matrices $Q_1$ and $Q_2$ are positive definite. Recall that $S_i$ is a positive definite matrix
and $S_i S_i = Q_i$. Consequently, one has
$$
  \mathcal{E}_i 
  = \Big\{ x_i \in \mathbb{R}^d \Bigm| \big\langle x_i - z_i, Q_i (x_i - z_i) \big\rangle \le 1 \Big\} 
  = \Big\{ x_i \in \mathbb{R}^d \Bigm| \| S_i (x_i - z_i) \|^2 \le 1 \Big\}.
$$
Denote $c_i = S_i z_i$ and $y_i = S_i x_i - c_i$, $i \in \{ 1, 2 \}$. Then problem
\eqref{prob:EllipsoidDist_Convex} can be rewritten as follows:
\begin{equation} \label{prob:EllipsoidsDist_ADMM}
  \min_{(x, y)} \: \frac{1}{2} \| x_1 - x_2 \|^2 \quad \text{s.t.} \quad
  S_i x_i - y_i = c_i, \quad \| y_i \| \le 1, \quad i \in \{ 1, 2 \}.
\end{equation}
Here $x = \left( \begin{smallmatrix} x_1 \\ x_2 \end{smallmatrix} \right) \in \mathbb{R}^{2d}$ and 
$y = \left( \begin{smallmatrix} y_1 \\ y_2 \end{smallmatrix} \right) \in \mathbb{R}^{2d}$. Note that problem
\eqref{prob:EllipsoidsDist_ADMM} is a particular case of problem \eqref{prob:SeparableConvexProblem} with
$f(x) = 0.5 \| x_1 - x_2 \|^2$, $g(y) \equiv 0$, and
\begin{gather*}
  A = \begin{pmatrix} S_1 & \mathbb{O}_d \\ \mathbb{O}_d & S_2 \end{pmatrix}, \quad
  B = \begin{pmatrix} -I_d & \mathbb{O}_d \\ \mathbb{O}_d & - I_d \end{pmatrix}, \quad
  c = \begin{pmatrix} c_1 \\ c_2 \end{pmatrix}, \\
  X = \mathbb{R}^{2d}, \quad
  Y = \left\{ y = \begin{pmatrix} y_1 \\ y_2 \end{pmatrix} \in \mathbb{R}^{2d} \biggm| 
  \| y_1 \| \le 1, \: \| y_2 \| \le 1 \right\},
\end{gather*}
where $\mathbb{O}_d$ is the zero matrix of order $d$ and $I_d$ is the identity matrix of order $d$. Therefore we can
apply the ADDM to find its solution. In the case of problem \eqref{prob:EllipsoidsDist_ADMM} this method takes the
following form:
\begin{align} \notag
  x^{n + 1} &= \argmin_{x_1, x_2 \in \mathbb{R}^d} \left\{ \frac{1}{2} \| x_1 - x_2 \|^2 
  + \sum_{i = 1}^2 \Big( - \langle \lambda_i^n, S_i x_i \rangle 
  + \frac{\tau}{2} \big\| S_i x_i - y_i^n - c_i \big\|^2 \Big) \right\}, 
  \\
  y_i^{n + 1} &= \argmin_{\| y_i \| \le 1} \Big( \langle \lambda_i^n, y_i \rangle
  + \frac{\tau}{2} \big\| S_i x_i^{n + 1} - y_i - c_i \big\|^2 \Big), \quad i \in \{1, 2 \} \label{prob:ADMM_y_def}
  \\
  \lambda_i^{n + 1} &= \lambda_i^n - \tau(S_i x_i^{n + 1} - y_i^{n + 1} - c_i), \quad i \in \{ 1, 2 \}.
  \notag
\end{align}
Observe that in order to find $x^{n + 1}$ one simply has to minimise the quadratic function
\begin{equation} \label{eq:QuadForm}
  f_n(x) = \frac{1}{2} \| x_1 - x_2 \|^2 + \sum_{i = 1}^2 \Big( - \langle \lambda_i^n, S_i x_i \rangle 
  + \frac{\tau}{2} \big\| S_i x_i - y_i^n - c_i \big\|^2 \Big).
\end{equation}
As is easily seen, the Hessian matrix of this function does not depend on $n$ and has the form
\begin{equation} \label{eq:QuadFormHessian}
  H(\tau) = \begin{pmatrix} I_d + \tau Q_1 & - I_d \\ - I_d & I_d + \tau Q_2 \end{pmatrix}.
\end{equation}
This matrix is positive definite, since
$$
  \langle x, H(\tau) x \rangle = \| x_1 - x_2 \|^2 + \tau \langle x_1, Q_1 x_1 \rangle + 
  \tau \langle x_2, Q_2 x_2 \rangle \quad \forall x \in \mathbb{R}^{2d}
$$
and the matrices $Q_i$ are positive definite. Therefore the point $x^{n + 1}$ is correctly defined. It can be
found, e.g. by the direct minimisation of the function $f_n$ via the conjugate gradient method or by solving the system
of linear equations $\nabla f_n(x^{n + 1}) = 0$ with a positive definite matrix, which has the form:
\begin{equation} \label{eq:LinearSystem_ConvexCase}
  \begin{pmatrix} I_d + \tau Q_1 & - I_d \\ - I_d & I_d + \tau Q_2 \end{pmatrix} x^{n + 1} =
  \begin{pmatrix} S_1 & \mathbb{O}_d \\ \mathbb{O}_d & S_2 \end{pmatrix} \Big( \lambda^n + \tau( y^n + c ) \Big).
\end{equation}
In our numerical experiments we computed $x^{n + 1}$ by directly solving this system with the use of the Cholesky
decomposition of the Hessian matrix $H(\tau)$. Below we describe an algorithm for finding the distance using this
approach, although it should be noted that other methods for computing $x^{n + 1}$ can be applied.

\begin{remark}
One can reduce the dimension of system \eqref{eq:LinearSystem_ConvexCase} by resolving it with respect to, say,
$x_2^{n + 1}$. Namely, from \eqref{eq:LinearSystem_ConvexCase} it follows that
$$
  x_2^{n + 1} = (I_d + \tau Q_1) x_1^{n + 1} - S_1 \big( \lambda_1^n + \tau (y_1^n + c_1) \big).
$$
Hence one gets the following system of linear equations for $x_1$:
\begin{equation} \label{eq:LinearSystem_Reduced}
\begin{split}
  \Big( \tau(Q_1 + Q_2) + \tau^2 Q_2 Q_1 \Big) x_1^{n + 1} &= S_2 \big( \lambda_2^n + \tau (y_2^n + c_2) \big)
  \\
  &+ (I_d + \tau Q_2) \Big[ S_1 \big( \lambda_1^n + \tau (y_1^n + c_1) \big) \Big].
\end{split}
\end{equation}
Note that the dimension of this systems is twice smaller than the dimension of system
\eqref{eq:LinearSystem_ConvexCase}. We performed numerical experiments on various test problems to verify whether it is
more efficient to solve the reduced system \eqref{eq:LinearSystem_Reduced} instead of system
\eqref{eq:LinearSystem_ConvexCase}. The results of these experiments demonstrated that the implementation of the ADMM
directly solving system \eqref{eq:LinearSystem_ConvexCase} outperformes the implementation based on the reduced
system \eqref{eq:LinearSystem_Reduced} on all test problems with both sparse and dense matrices $Q_1$ and $Q_2$.
However, this effect might be due to some peculiarities of our implementation of the ADMM in \textsc{Matlab}, and other
implementations may lead to different results.
\end{remark}

Observe that problem \eqref{prob:ADMM_y_def} for computing the next estimate of $y$ is equivalent to the following
optimisation problem:
$$
  \min_{\| y_i \| \le 1} 
  \frac{\tau}{2} \Big\| y_i - S_i x_i^{n + 1} + c_i + \frac{1}{\tau} \lambda_i^n \Big\|^2,
  \quad i \in \{ 1, 2 \}
$$
Thus, $y_i^{n + 1}$ is nothing but the projection of the point $S_i x_i^{n + 1} - c_i - (1 / \tau) \lambda_i^n$ onto the
unit ball centered at the origin, which can be easily computed analytically. 

To define a stopping criterion for the method, let us consider optimality conditions for problem
\eqref{prob:EllipsoidsDist_ADMM} and more general problem \eqref{prob:SeparableConvexProblem}. The KKT optimality
conditions for problem \eqref{prob:SeparableConvexProblem} have the form:
\begin{gather*}
  x^* \in X, \quad y^* \in Y, \quad A x^* + B y^* = c, 
  \\ 
  \langle \nabla f(x^*) - A^T \lambda^*, x - x^* \rangle \ge 0 \quad \forall x \in X, 
  \\
  \langle \nabla g(y^*) - B^T \lambda^*, y - y^* \rangle \ge 0 \quad \forall y \in Y.
\end{gather*}
In the case of problem \eqref{prob:EllipsoidsDist_ADMM} these conditions take the form:
\begin{gather} 
  \| y_i^* \| \le 1, \quad S_i x_i^* - y_i^* = c_i, \quad i \in \{ 1, 2 \}, 
  \label{eq:OptmCond_Feasibility}\\
  x_1^* - x_2^* - S_1 \lambda_1^* = 0, \quad
  x_2^* - x_1^* - S_2 \lambda_2^* = 0, 
  \\
  \langle \lambda_i^*, y_i - y_i^* \rangle \ge 0 \quad \forall y_i \colon \| y_i \| \le 1, \quad i \in \{ 1, 2 \}.
  \label{eq:OptmCond_y}
\end{gather}
Note that conditions \eqref{eq:OptmCond_y} are satisfied iff $y_i^* = \proj_{B^d}(y_i^* - \lambda_i^*)$, 
$i \in \{ 1, 2 \}$, where $\proj_{B^d}(\cdot)$ is the Euclidean projection onto the unit ball $B^d$ in $\mathbb{R}^d$
centered at the origin. Indeed, by definition $y_i^* = \proj_{B^d}(y_i^* - \lambda_i^*)$ iff $y_i^*$ is an optimal
solution of the convex problem
$$
  \min \: \frac{1}{2} \| y_i - (y_i^* - \lambda_i^*) \|^2
  \quad \text{subject to} \quad \| y_i \| \le 1.
$$
Bearing in mind the fact that the necessary and sufficient optimality conditions for this problem coincide with 
\eqref{eq:OptmCond_y} we arrive at the required result. Let us note that the equivalence between the optimality
condition of the form \eqref{eq:OptmCond_y} and the corresponding equality for the projection was pointed out, e.g. in
\cite{JiaCaiHan}.

The violation of optimality conditions \eqref{eq:OptmCond_Feasibility}--\eqref{eq:OptmCond_y} can be measured with the
use of the following functions:
\begin{gather*}
  R_x(x, y, \lambda) = \begin{pmatrix} x_1 - x_2 - S_1 \lambda_1 \\ x_2 - x_1 - S_2 \lambda_2 \end{pmatrix}, \quad
  R_y(x, y, \lambda) = 
  \begin{pmatrix} y_1 - \proj_{B^d}(y_1 - \lambda_1) \\ y_2 - \proj_{B^d}(y_2 - \lambda_2) \end{pmatrix},  \\
  R_c(x, y, \lambda) = \begin{pmatrix} S_1 x_1 - y_1 - c_1 \\ S_2 x_2 - y_2 - c_2 \end{pmatrix}.
\end{gather*}
The function $R_c(x, y, \lambda)$ measures the violation of the equality constraints of problem
\eqref{prob:EllipsoidsDist_ADMM}. The value 
$\| R_x(x, y, \lambda) \| + \| R_y(x, y, \lambda \| + \| R_c(x, y, \lambda) \|$ indicates how far a point
$(x, y)$ lies from an optimal solution of problem \eqref{prob:EllipsoidsDist_ADMM}, which implies that the condition
\begin{equation} \label{eq:StoppingCriterion}
  \| R_x(x^{n + 1}, y^{n + 1}, \lambda^{n + 1}) \| + \| R_y(x^{n + 1}, y^{n + 1}, \lambda^{n + 1}) \| + 
  \| R_c(x^{n + 1}, y^{n + 1}, \lambda^{n + 1}) \| < \varepsilon
\end{equation}
can be used as a stopping criterion. Let us note that different stopping criteria, such as the standard one based on
primal-dual residuals \cite[Section~3.3]{BoydParikhChuEtAl}, can be used as well. We opted for criterion
\eqref{eq:StoppingCriterion}, since similar stopping criteria for the ADMM performed well in a number of numerical
experiments presented in \cite{JiaCaiHan}. 

Thus, we arrive at the following scheme of the ADMM (given in Algorithm~\ref{alg:ADMM_convex}) for solving problem
\eqref{prob:EllipsoidsDist_ADMM}, which is equivalent to the original problem of finding the distance between 
ellipsoids. 

\begin{algorithm}[t]	\label{alg:ADMM_convex}
\caption{The ADMM for finding the distance between ellipsoids.}

\noindent\textbf{Step~1.} {Choose $y^0, \lambda^0 \in \mathbb{R}^{2d}$, $\tau > 0$, and $\varepsilon > 0$, and set 
$n := 0$.}

\noindent\textbf{Step~2.} {Compute the square roots $S_i$ of the matrices $Q_i$, vectors $c_i = S_i z_i$, 
$i \in \{1, 2 \}$, and compute the Cholesky decomposition $H(\tau) = L L^T$ of the matrix $H(\tau)$ given in
\eqref{eq:QuadFormHessian}.}

\noindent\textbf{Step~3.} {Compute 
$u^n = \left(\begin{smallmatrix} S_1(\lambda_1^n + \tau(y_1^n + c_1)) \\ S_2(\lambda_2^n + \tau(y_1^n + c_1))
\end{smallmatrix}\right)$ and solve the linear systems $L w^n = u^n$ and $L^T x^{n + 1} = w^n$.
}

\noindent\textbf{Step~4.} {Compute $v_i^n = S_i x_i^{n + 1} - c_i - (1 / \tau) \lambda_i^n$, $i \in \{ 1, 2 \}$. If 
$\| v_i^n \| \le 1$, then $y_i^{n + 1} = v_i^n$; otherwise, $y_i^{n + 1} = \frac{1}{\| v_i^n \|} v_i^n$, 
$i \in \{1, 2 \}$.
}

\noindent\textbf{Step~5.} {Put $\lambda_i^{n + 1} = \lambda_i^n - \tau(S_i x_i^{n + 1} - y_i^{n + 1} - c_i)$, 
$i \in \{ 1, 2 \}$.
}

\noindent\textbf{Step~6.} {Compute $R_x(x^{n + 1}, y^{n + 1}, \lambda^{n + 1})$, 
$R_y(x^{n + 1}, y^{n + 1}, \lambda^{n + 1})$, and $R_c(x^{n + 1}, y^{n + 1}, \lambda^{n + 1})$. 
If condition \eqref{eq:StoppingCriterion} holds true, then \textbf{Stop}. Otherwise, set $n := n + 1$ and go to
\textbf{Step 3}.
}
\end{algorithm}

Observe that in Algorithm~\ref{alg:ADMM_convex} the Cholesky decomposition $H(\tau) = L L^T$ is computed only once
before the main iterations of the method start. The use of this decomposition allows one to significantly reduce the
cost of computing the next iterate $x^{n + 1}$ in comparison with direct solution of the corresponding system of linear
equation. Furthermore, we chose the Cholesky decomposition for solving the corresponding linear systems because of its
well-known efficiency and numerical stability. However, the downside of this approach is the fact that one has to store
the lower triangular matrix $L$ of order $2d$ throughout iterations. Therefore, for large-scale problems a direct
minimization of the corresponding quadratic function might be more efficient due to reduced memory consumption.

Let us also describe a version of the ADMM for finding the distance between ellipsoids with automatic adjustments of
the penalty parameter $\tau$. The rule for modifying $\tau$ that we use was proposed in \cite{HeYangWang}. Numerical 
experiments on various problems \cite{HeYangWang,JiaCaiHan} clearly demonstrated that the ADMM using this rule
significantly outperforms the standard version of the ADMM with fixed $\tau$. Let us also note that the rule for
updating the penalty parameter from \cite{HeYangWang} is defined via an arbitrary sequence 
$\{ \alpha_n \} \subset [0, + \infty)$, and the convergence of the ADMM is established in \cite{HeYangWang} under the
assumption that $\sum_{n = 0}^{\infty} \alpha_n < + \infty$.

A version of the ADMM with automatic adjustments of the penalty parameter is given in
Algorithm~\ref{alg:ADMM_sa_convex}.

\begin{algorithm}[t]	\label{alg:ADMM_sa_convex}
\caption{The ADMM with penalty adjustments.}

\noindent\textbf{Step~1.} {Choose $y^0, \lambda^0 \in \mathbb{R}^{2d}$, $\tau_0 > 0$, $\eta \in (0, 1)$, a sequence 
$\{ \alpha_n \} \subset [0, + \infty)$, and $\varepsilon > 0$. Set $n := 0$ and \textbf{flag} ${}={}$ \textbf{true}.}

\noindent\textbf{Step~2.} {Compute the square roots $S_i$ of the matrices $Q_i$ and vectors $c_i = S_i z_i$, 
$i \in \{1, 2 \}$.}

\noindent\textbf{Step~3.} {If \textbf{flag} ${}={}$ \textbf{true}, compute the Cholesky decomposition 
$H(\tau_n) = L L^T$ of the matrix $H(\tau_n)$ given in \eqref{eq:QuadFormHessian} and put 
\textbf{flag} ${}={}$ \textbf{false}.}

\noindent\textbf{Step~4.} {Compute 
$u^n = \left(\begin{smallmatrix} S_1(\lambda_1^n + \tau_n(y_1^n + c_1)) \\ S_2(\lambda_2^n + \tau_n(y_1^n + c_1))
\end{smallmatrix}\right)$ and solve the linear systems $L w^n = u^n$ and $L^T x^{n + 1} = w^n$.
}

\noindent\textbf{Step~5.} {Compute $v_i^n = S_i x_i^{n + 1} - c_i - (1 / \tau_n) \lambda_i^n$, $i \in \{ 1, 2 \}$. If 
$\| v_i^n \| \le 1$, then $y_i^{n + 1} = v_i^n$; otherwise, $y_i^{n + 1} = \frac{1}{\| v_i^n \|} v_i^n$, 
$i \in \{1, 2 \}$.
}

\noindent\textbf{Step~6.} {Put $\lambda_i^{n + 1} = \lambda_i^n - \tau_n(S_i x_i^{n + 1} - y_i^{n + 1} - c_i)$, 
$i \in \{ 1, 2 \}$.
}

\noindent\textbf{Step~7.} {Compute $R_x(x^{n + 1}, y^{n + 1}, \lambda^{n + 1})$, 
$R_y(x^{n + 1}, y^{n + 1}, \lambda^{n + 1})$, and $R_c(x^{n + 1}, y^{n + 1}, \lambda^{n + 1})$. 
If condition \eqref{eq:StoppingCriterion} holds true, then \textbf{Stop}. 
}

\noindent\textbf{Step~8.} {Define
$$
  \tau_{n + 1} = \begin{cases}
    \tau_n (1 + \alpha_n), & \text{if } 
    \| R_x(x^{n + 1}, y^{n + 1}, \lambda^{n + 1}) \| < \eta \| R_c(x^{n + 1}, y^{n + 1}, \lambda^{n + 1}) \|, \\
    \tau_n / (1 + \alpha_n), & \text{if } 
    \eta \| R_x(x^{n + 1}, y^{n + 1}, \lambda^{n + 1}) \| > \| R_c(x^{n + 1}, y^{n + 1}, \lambda^{n + 1}) \|, \\
    \tau_n, & \text{otherwise.}
  \end{cases}
$$
If $\tau_{n + 1} \ne \tau_n$, set \textbf{flag} ${}={}$ \textbf{true}. Put $n := n + 1$ and go to \textbf{Step 3}.
}
\end{algorithm}

The only significant difference between Algorithms~\ref{alg:ADMM_convex} and \ref{alg:ADMM_sa_convex} is the additional
step (Step 8) describing adjustments of the penalty parameter $\tau_n$. However, note that in
Algorithm~\ref{alg:ADMM_sa_convex}, in contrast to Algorithm~\ref{alg:ADMM_convex}, one needs to recompute the Cholesky
decomposition of the Hessian matrix \eqref{eq:QuadFormHessian} every time the penalty parameter is adjusted. Results
of numerical experiments given in Section~\ref{subsect:NumericalExperiments_ConvexCase} demonstrate that (i) the
adjustments of the penalty parameter allow one to substantially reduce the number of iterations of the ADMM needed to
find an optimal solution with a prescribed tolerance, and (ii) the benefits of adjusting the penalty parameter
largely outweigh the extra time needed to recompute the Cholesky decomposition of the Hessian matrix multiple times. 

We used the boolean variable ``flag'' to indicate whether an update of the Cholesky decomposition is needed. It helps
one to avoid unnecessary computations by making the algorithm recompute the decomposition only when the penalty
parameter has been changed.

\begin{remark}
In our numerical experiments with ill-conditioned problems, stopping criterion \eqref{eq:StoppingCriterion}
was sometimes too optimistic in the sense that it led to the termination of the algorithm before the desired accuracy
of solution was achieved. To overcome this issue, in our implementations of Algorithms~\ref{alg:ADMM_convex} and
\ref{alg:ADMM_sa_convex} we additionally checked the conditions
\begin{equation} \label{eq:AdditionalStoppingCriterion}
  \big| \langle x_i^{n + 1} - z_i, Q_i(x_i^{n + 1} - z_i) \rangle - 1 \big| < \varepsilon,
  \quad i \in \{ 1, 2 \},
\end{equation}
if condition \eqref{eq:StoppingCriterion} was satisfied and $\| x_1^{n + 1} - x_2^{n + 1} \| > \delta$ for some small
$\delta > 0$. Let us note that inequalities \eqref{eq:AdditionalStoppingCriterion} simply mean that the points 
$x_i^{n + 1}$ lie close to the boundaries of the corresponding ellipsoids. Condition
\eqref{eq:AdditionalStoppingCriterion} is based on the fact that if the ellipsoids $\mathcal{E}_1$ and $\mathcal{E}_2$
do not intersect, then a unique solution of problem \eqref{prob:EllipsoidDist_Convex} necessarily lies on the boundaries
of these ellipsoids.

The additional stopping criterion \eqref{eq:AdditionalStoppingCriterion} led to noticeably improved results in most
cases when the stopping criterion \eqref{eq:StoppingCriterion} failed to recognise the non-optimality of the current
iterate. Alternatively, decreasing $\varepsilon$ by a factor of $100$ we obtained the same or even better improvements
in the accuracy of the found solution without any significant increase in the run time. 
\end{remark}

\subsection{Analysis of the methods}
\label{subsect:ConvergenceConvexCase}

Let us briefly discuss convergence of the two methods for finding the distance between ellipsoids described in the 
previous section. Note that it is sufficient to consider only Algorithm~\ref{alg:ADMM_sa_convex}, since this algorithm
is reduced to Algorithm~\ref{alg:ADMM_convex} when $\alpha_n \equiv 0$. 

\begin{theorem}
For any choice of parameters $y^0, \lambda^0 \in \mathbb{R}^{2d}$, $\tau_0 > 0$, $\eta \in (0, 1)$,
$\{ \alpha_n \} \subset [0, + \infty)$, $\sum_{n = 0}^{\infty} \alpha_n < + \infty$, and $\varepsilon > 0$
Algorithm~\ref{alg:ADMM_sa_convex} is well-defined and terminates after a finite number of iterations. Moreover, if
$\varepsilon = 0$, then Algorithm~\ref{alg:ADMM_sa_convex} generates a sequence $\{ (x^n, y^n, \lambda^n) \}$ that
converges to a point $(x^*, y^*, \lambda^*)$ satisfying optimality conditions
\eqref{eq:OptmCond_Feasibility}--\eqref{eq:OptmCond_y} and such that $x^*$ is a globally optimal solution of problem
\eqref{prob:EllipsoidDist_Convex}.
\end{theorem}

\begin{proof}
The fact that Algorithm~\ref{alg:ADMM_sa_convex} is well-defined follows directly from its detailed description given
in the previous subsection. Indeed, Step~2 of Algorithm~\ref{alg:ADMM_sa_convex} is well-defined, since by our
assumption the matrices $Q_i$, $i \in \{ 1, 2 \}$, are positive definite and any positive definite matrix has a unique
square root. Steps~3 and 4 are well-defined, since, as was shown above, the matrix $H(\tau)$ is positive definite for
any $\tau > 0$ and $\tau_n > 0$ for all $n \in \mathbb{N}$ according to Step~8. The correctness of all other steps of
the algorithm is obvious.

Let us prove that in the case $\varepsilon = 0$ Algorithm~\ref{alg:ADMM_sa_convex} generates a sequence 
$\{ (x^n, y^n, \lambda^n) \}$ that converges to a point $(x^*, y^*, \lambda^*)$ satisfying optimality conditions
\eqref{eq:OptmCond_Feasibility}--\eqref{eq:OptmCond_y}. Then the finite termination of this algorithm in the case
$\varepsilon > 0$ follows directly from the definition of stopping criterion \eqref{eq:StoppingCriterion}. Moreover,
from the fact that problem \eqref{prob:EllipsoidsDist_ADMM} is convex it follows that $(x^*, y^*)$ is a globally
optimal solution of this problem. Consequently, $x^*$ is a globally optimal solution of 
problem \eqref{prob:EllipsoidDist_Convex}, since by construction $(x^*, y^*)$ is a globally optimal solution of problem
\eqref{prob:EllipsoidsDist_ADMM} if and only if $x^*$ is a globally optimal solution of problem 
\eqref{prob:EllipsoidDist_Convex} and $S_i(x_i^* - z_i) = y_i^*$.

To prove the convergence of the sequence $\{ (x^n, y^n, \lambda^n) \}$ to a KKT point of problem
\eqref{prob:EllipsoidsDist_ADMM}, note that this problem is obviously equivalent to the problem of finding a solution
$(x^*, y^*)$ of the following variational inequality:
\begin{equation} \label{prob:VariationalInequal}
\begin{split}
  &\langle x - x^*, F(x^*) \rangle + \langle y - y^*, G(y^*) \rangle \ge 0 
  \\
  &\forall (x, y) \in \mathbb{R}^{2d} \times \mathbb{R}^{2d} \colon
  S_i x_i - y_i = c_i, \quad \| y_i \| \le 1, \quad i \in \{ 1, 2 \}, 
  \\
  &S_i x_i^* - y_i^* = c_i, \quad \| y_i^* \| \le 1, \quad i \in \{ 1, 2 \}.
\end{split}
\end{equation}
Here $F(x) = \left( \begin{smallmatrix} x_1 - x_2 \\ x_2 - x_1 \end{smallmatrix} \right)$ and $G(y) \equiv 0$. Note
that the mappings $F$ and $G$ are monotone as the gradients of convex functions.

As one can easily verify, Algorithm~\ref{alg:ADMM_sa_convex} (more precisely, algorithm described in
\eqref{prob:ADMM_y_def}) is a particular case of the ADMM from \cite{HeYangWang} applied to the variational inequality
\eqref{prob:VariationalInequal}. By \cite[Therorem~4.1 and Remark~4.2]{HeYangWang} the sequence generated by the ADMM
for solving problem \eqref{prob:VariationalInequal} converges to a solution of this problem for all values of
parameters, provided $\sum_{n = 0}^{\infty} \alpha_n < + \infty$, the algorithm is well-defined, and there exists a
solution of the variational inequality \eqref{prob:VariationalInequal}. 

The fact that the algorithm is well-defined was verified above. Note also that problem \eqref{prob:EllipsoidDist_Convex}
and the equivalent problem \eqref{prob:EllipsoidsDist_ADMM} have globally optimal solutions, since the feasible region
of problem \eqref{prob:EllipsoidDist_Convex} is obviously compact. In turn, any globally optimal solution of problem
\eqref{prob:EllipsoidsDist_ADMM} is a solution of the variational inequality \eqref{prob:VariationalInequal} due to the
convexity of problem \eqref{prob:EllipsoidsDist_ADMM}. Therefore one can conclude that a solution of the variational
inequality \eqref{prob:VariationalInequal} exists. 

Thus, by \cite[Therorem~4.1 and Remark~4.2]{HeYangWang} the sequence $\{ (x^n, y^n, \lambda^n) \}$ generated by
Algorithm~\ref{alg:ADMM_sa_convex} converges to a solution of the variational inequality \eqref{prob:VariationalInequal}
for all values of the parameters, provided  $\sum_{n = 0}^{\infty} \alpha_n < + \infty$. Any solution of
\eqref{prob:VariationalInequal} is a globally optimal solution of problem \eqref{prob:EllipsoidsDist_ADMM}.
Therefore the limit point $(x^*, y^*, \lambda^*)$ must satisfy optimality conditions
\eqref{eq:OptmCond_Feasibility}--\eqref{eq:OptmCond_y}, since Slater's condition obviously holds true for 
problem \eqref{prob:EllipsoidsDist_ADMM} (namely, put $x_i = z_i$ and $y_i = 0$).
\end{proof}

\subsection{Numerical experiments}
\label{subsect:NumericalExperiments_ConvexCase}

Without trying to present a thorough comparative analysis of Algorithms~\ref{alg:ADMM_convex} and
\ref{alg:ADMM_sa_convex} and other existing methods for finding the distance between two ellipsoids on various problem
instances (e.g. ellipsoids lying very far apart or very close to each other, `flat' and `elongated' ellipsoids, i.e. one
of the eigenvalues of the matrix $Q_i$ is much smaller/greater than others, etc.), let us give some results of
preliminary numerical experiments demonstrating the higher efficiency of the ADMM in comparison with other existing
methods for finding the distance between ellipsoids.

To generate the problem data for numerical experiments, first we randomly generated matrices $A_i$, $i \in \{ 1, 2 \}$,
of dimension $d \in \mathbb{N}$, whose elements were uniformly distributed in the interval $[-10, 10]$. If the matrix
$A_i$ has full rank, we define $Q_i = A_i^T A_i$. Otherwise, the matrix $A_i$ was randomly generated again till it had
full rank, to ensure that the matrix $Q_i$ is positive definite. The centres $z_i$ of the ellipsoids were also randomly
generated in such a way that their coordinates are uniformly distributed in the same interval $[-10, 10]$. 

The parameters of Algorithms~\ref{alg:ADMM_convex} and \ref{alg:ADMM_sa_convex} were chosen as follows:
$y^0 = \lambda^0 = 0$, $\varepsilon = 10^{-6}$, $\tau = \tau_0 = 1$. As in paper \cite{HeYangWang}, where the penalty
adapting strategy was proposed, we put $\eta = 0.1$ and $\alpha_n = 1$, if $n < 100$, while $\alpha_n = 0$, 
if $n \ge 100$ in Algorithm~\ref{alg:ADMM_sa_convex}. Below, Algorithm~\ref{alg:ADMM_convex} is denoted as ADMM, while
Algorithm~\ref{alg:ADMM_sa_convex} is denoted as sa-ADMM (self-adaptive ADMM; see \cite{HeYangWang}).

We compared Algorithms~\ref{alg:ADMM_convex} and \ref{alg:ADMM_sa_convex} with Lin and Han's method \cite{LinHan}, 
the exact penalty method from \cite{TamasyanChumakov}, and the so-called charged balls method \cite{Abbasov}. The
starting points in Lin and Han's (LH) algorithm were chosen as the centres $z_1$ and $z_2$ of the ellipsoids. The
parameters $\gamma_1$ and $\gamma_2$ were chosen as $\gamma_i = 1 / \| Q_i \|_1$, where $\| \cdot \|_1$ is the $1$-norm
(i.e. the maximum absolute column sum of the matrix), and the inequalities $\theta_1 < 10^{-6}$ and $\theta_2 < 10^{-6}$
were used as a stopping criterion. 

The penalty parameter $\lambda$ from the exact penalty method \cite{TamasyanChumakov} was defined as $\lambda = 100$. We
used the inequality $\| G^*(z_k) \| < \varepsilon = 10^{-5}$ as a stopping criterion, since in some cases the algorithm
failed to terminate when the value $\varepsilon = 10^{-6}$ was used. Since no rules for choosing the starting points
were given in \cite{TamasyanChumakov} and the exact penalty method cannot start from the centres of the ellipsoids, we
tried using two different initial guesses. The first rule for choosing starting points consisted in setting 
$x_1^0 = z_1$ and defining $x_2^0$ as the vector $z_2$ perturbed by some other vector with small coordinates, i.e.
$x_2^0 = z_2 + e$. We chose $e = (0.1, \ldots, 0.1)^T \in \mathbb{R}^d$. We denote the exact penalty method using this
rule as $EP_1$. The second rule consisted in defining $x_1^0$ and $x_2^0$ as the points at which the segment 
$\co\{ z_1, z_2 \}$ intersect the boundaries of the ellipsoids. We denote the exact penalty method with these starting
points as $EP_2$.

Finally, for the charged balls method \cite{Abbasov} (denoted CB) we used the same parameters as given in
\cite[Section~3]{Abbasov}. Namely, we set $p_1 = 10$, $p_2 = 1$, $\delta = 0.1$, and $\varepsilon = 10^{-8}$. The
starting points were chosen in the same way as in $EP_2$ method, since in our experiments the charged balls method
produced incorrect results, when the starting points were lying in the interiors of the ellipsoids. 

All algorithms were implemented in \textsc{Matlab}. We terminated the algorithms if the number of iterations exceeded
$10^7$ for the charged balls method and $10^6$ for all other methods. Since the behaviour of the algorithms appeared to
be very dependent on a particular problem instance (i.e. a particular algorithm might be very slow on one randomly
generated problem and very fast on another), we generated 10 problems for a given dimension 
$d \in \{ 10, 20, 30, 50, 100, 200, 300, 500, 1000, 2000 \}$ and run each algorithm on these 10 problems. The total
run time of each method rounded to the nearest tenth is presented in Figure~\ref{fig:ConvexCase} and
Table~\ref{tab:ConvexCase}. 

\begin{figure}[t]
\centering
\includegraphics[width=0.6\linewidth]{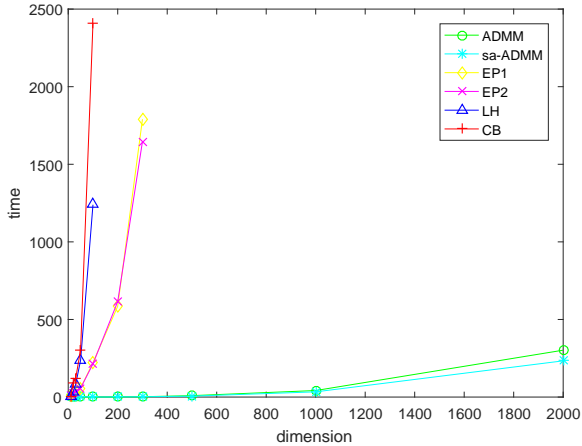}
\caption{The run time of each method for $d \in \{ 10, 20, 30, 50, 100, 200, 300, 500, 1000, 2000 \}$.}
\label{fig:ConvexCase}
\end{figure}

\begin{table} 
\begin{tabular}{c c c c c c c} 
  \hline
  d & ADMM & sa-ADMM & LH & $EP_1$ & $EP_2$ & CB \\
  \hline
  10 & 0.1 & 0.1 & 3.9 & 7.4 & 10.4 & 6.6 \\
  \hline
  20 & 0.1 & 0.1 & 33.2 & 26 & 27 & 92.1 \\
  \hline
  30 & 0.1 & 0.1 & 81 & 37.4 & 32.6 & 118.8 \\
  \hline
  50 & 0.1 & 0.1 & 236.6 & 57.5 & 57.5 & 303.6 \\
  \hline
  100 & 0.2 & 0.2 & 1239.3 & 222.1 & 213.8 & 2408.6 \\
  \hline
  200 & 0.9 & 0.7 & --- & 585 & 615.4 & --- \\
  \hline
  300 & 2.1 & 1.8 & --- & 1790.5 & 1646.6 & --- \\
  \hline
  500 & 9.2 & 5.9 & --- & --- & --- & --- \\
  \hline
  1000 & 42.1 & 33.6 & --- & --- & --- & --- \\
  \hline
  2000 & 302.1 & 233.4 & --- & --- & --- & --- \\
  \hline
\end{tabular}
\caption{The run time of each method in seconds. `---' indicates that either the time exceeded one hour or the
algorithm failed to find a solution due to reaching the prespecified maximal number of iterations.}
\label{tab:ConvexCase}
\end{table}

\begin{table} 
\begin{tabular}{c c c c c c c c c c c} 
  \hline
  d & 10  & 20 & 30 & 50 & 100 & 200 & 300 & 500 & 1000 & 2000 \\
  \hline
  ADMM & 45.3 & 153.3 & 113.2 & 154.5 & 152.3 & 244.2 & 328.4 & 433.4 & 425.8 & 761.5 \\
  \hline
  sa-ADMM & 46.6 & 128.4 & 113.2 & 120.1 & 108.2 & 213.7 & 273.9 & 263.4 & 321.1 & 557.6 \\
  \hline
\end{tabular}
\caption{The average number of iterations till termination for Algorithms~\ref{alg:ADMM_convex} and
\ref{alg:ADMM_sa_convex}.}
\label{tab:IterNumber}
\end{table}

The results of our numerical experiments clearly demonstrate that Algorithms~\ref{alg:ADMM_convex} and
\ref{alg:ADMM_sa_convex} considerably outperform all other existing methods for finding the distance between two
ellipsoids and are more suitable for large dimensional problems than other methods. 

Let us give some comments about the performance of other methods. Firstly, numerical experiments showed that on average
there is no significant difference between the exact penalty methods $EP_1$ and $EP_2$. Thus, it is not clear how to
initialise the exact penalty method \cite{TamasyanChumakov} to improve its performance.

Secondly, one should point out that the charged balls method, in accordance with the example given in \cite{Abbasov},
produced the most inaccurate results among all methods. In our numerical experiments the distance computed by the
charged balls method was at least $10^{-3}$ (in some examples even $10^{-2}$) greater than the distance computed by
other methods. We tried decreasing $\varepsilon$, but it did not improve the results. By defining the stepsize $\delta$
in the charged balls method as $\delta = 10^{-3}$ (or a smaller value) we managed to obtain the same accuracy as with
the use of other methods. However, such choice of the stepsize increased the run time by more than 10 times. That is why
we do not report the results of numerical experiments with different stepsizes $\delta$ here. Moreover, let us note
that, as was pointed out in \cite{Abbasov}, by changing the parameters $p_1$ and $p_2$ in the charged balls method one
could significantly improve its performance. We tried to find optimal values of these parameters; however, it turned out
that a nearly optimal choice of the parameters for one problem often became a very poor choice for another problem. That
is why we used the same values of the parameters $p_1$ and $p_2$ as in \cite{Abbasov}.

Finally, the results of numerical experiments showed that Algorithm~\ref{alg:ADMM_sa_convex} with adjustable penalty
parameter is faster than the ADMM with the fixed penalty parameter, despite the fact that every change of the penalty
parameter requires a recomputation of the Cholesky decomposition of the matrix $H(\tau_n)$ of dimension $2d$. 
The reason behind this lies in the fact that penalty adjustments allow one to substantially reduce the number of
iterations of the algorithm. To illustrate this point, in Table~\ref{tab:IterNumber} we present the average number of
iterations till termination for Algorithms~\ref{alg:ADMM_convex} and \ref{alg:ADMM_sa_convex}.

\section{The distance between ellipsoids: the nonconvex case}
\label{sect:NonconvexCase}

In the second part of the paper we study the problem of finding the distance between the boundaries of the two
ellipsoids $\mathcal{E}_1$ and $\mathcal{E}_2$ in $\mathbb{R}^d$, defined in \eqref{eq:EllipsoidsDef}. One can readily
check that the boundary of the ellipsoid $\mathcal{E}_i$ is defined by the corresponding equality constraint:
$$
  \bd \mathcal{E}_i = \Big\{ x \in \mathbb{R}^d \Bigm| \big\langle x - z_i, Q_i (x - z_i) \big\rangle = 1 \Big\}, 
  \quad i \in \{ 1, 2 \}.
$$
Therefore, the problem of finding the distance between the boundaries of the ellipsoids is a \textit{nonconvex}
programming problem that can formalised as follows:
\begin{equation} \label{prob:EllipsoidDist_NonConvex}
\begin{split}
  &\min\: \frac{1}{2} \| x_1 - x_2 \|^2 \\ 
  & \text{subject to} \quad
  \langle x_1 - z_1, Q_1 (x_1 - z_1) \rangle = 1, \quad 
  \langle x_2 - z_2, Q_2 (x_2 - z_2) \rangle = 1.
\end{split}
\end{equation}
Since this problem has exactly the same structure as problem \eqref{prob:EllipsoidDist_Convex}, it is natural to
extend the alternating direction method of multipliers for solving problem \eqref{prob:EllipsoidDist_Convex}
(Algorithms~\ref{alg:ADMM_convex} and \ref{alg:ADMM_sa_convex}) to the case of problem
\eqref{prob:EllipsoidDist_NonConvex}.

\subsection{The alternating direction method of multipliers}
\label{subsect:ADMM_Nonconvex}

In order to apply the ADMM to problem \eqref{prob:EllipsoidDist_NonConvex} let us us rewrite this problem as an
optimisation problem of the form \eqref{prob:SeparableConvexProblem}. To this end, as in
Section~\ref{subsect:ADMM_ConvexCase}, let $S_i$ be the square root of the matrix $Q_i$. Then
$$
  \bd \mathcal{E}_i 
  = \Big\{ x_i \in \mathbb{R}^d \Bigm| \big\langle x_i - z_i, Q_i (x_i - z_i) \big\rangle = 1 \Big\} 
  = \Big\{ x_i \in \mathbb{R}^d \Bigm| \| S_i (x_i - z_i) \|^2 = 1 \Big\}.
$$
Denote $c_i = S_i z_i$ and $y_i = S_i x_i - c_i$, $i \in \{ 1, 2 \}$. Then problem
\eqref{prob:EllipsoidDist_NonConvex} can be rewritten as follows:
\begin{equation} \label{prob:EllipsoidsDistNonconvex_ADMM}
  \min_{(x, y)} \: \frac{1}{2} \| x_1 - x_2 \|^2 \quad \text{s.t.} \quad
  S_i x_i - y_i = c_i, \quad \| y_i \| = 1, \quad i \in \{ 1, 2 \}.
\end{equation}
This problem is very similar to problem \eqref{prob:EllipsoidsDist_ADMM} with the only difference being the fact that
the inequality constraints $\| y_i \| \le 1$ were replaced by the corresponding equality constraints $\| y_i \| = 1$,
$i \in \{ 1, 2 \}$. Therefore the only change one has to make in the ADMM for problem \eqref{prob:EllipsoidsDist_ADMM}
in order to apply it to problem \eqref{prob:EllipsoidsDistNonconvex_ADMM} consists in the way one computes the next
iterate $y_i^{n + 1}$. Namely, one has to define $y_i^{n + 1}$ as a projection of the point 
$S_i x_i^{n + 1} - c_i - (1 / \tau) \lambda_i^n$ onto the unit sphere, not the unit ball. Note, however, that in 
the case when $S_i x_i^{n + 1} - c_i - (1 / \tau) \lambda_i^n = 0$ this projection is not unique and one can define
$y_i^{n + 1}$ as any point from the unit sphere.

It should be noted that since problem \eqref{prob:EllipsoidsDistNonconvex_ADMM} is nonconvex, one cannot expect 
the ADMM for solving this problem to converge for all values of the penalty parameter $\tau$. Furthermore, the penalty
adjustments strategy from \cite{HeYangWang} becomes inviable. To define a rule for updating the penalty parameter that
ensures convergence, observe that the ADMM can be viewed as a modification of the classical augmented Lagrangian methods
based on the Hestenes-Powell-Rockafellar augmented Lagrangian \cite{BirginMartinez} to the case of problems of the form
\eqref{prob:SeparableConvexProblem}. Therefore, it is natural to adopt the rule for updating $\tau$ similar to 
the one used in general augmented Lagrangian methods \cite{LuoSunWu2008,LuoSunLi2007,WangLi2009,BirginMartinez}. 

Since the ADMM performs two separate steps ($x$-step and $y$-step) on every iteration, from a purely theoretical
point of view it seems natural to use the following inequality as a criterion for updating the penalty parameter:
\begin{multline} \label{eq:PenaltyUpdate_Nonconvex}
  \max_{i \in \{ 1, 2 \}} \Big\{
  \| S_i x_i^{n + 1} - y_i^n - c_i \| - \eta \| S_i x_i^n - y_i^{n - 1} - c_i \|,
  \\
  \| S_i x_i^{n + 1} - y_i^{n + 1} - c_i \| - \eta \| S_i x_i^n - y_i^n - c_i \| \Big\} > 0
\end{multline}
(here $\eta \in (0, 1)$ is a fixed parameter). Namely, if this inequality is satisfied, one defines 
$\tau_{n + 1} = \beta \tau_n$ for some $\beta > 1$. Otherwise, one sets $\tau_{n + 1} = \tau_n$. As we will show
in the following section, this rule significantly simplifies convergence analysis and under some additional assumptions
guarantees that the sequence $\{ (x^n, y^n, \lambda^n) \}$ converges to a KKT point of problem 
\eqref{prob:EllipsoidsDistNonconvex_ADMM}. However, this penalty updating rule performed very poorly in our numerical
experiments. The experiments showed that the sequence $\| S_i x_i^{n + 1} - y_i^n - c_i \|$ is not monotone at initial
stages, which in accordance to \eqref{eq:PenaltyUpdate_Nonconvex} results in a rapid increase of the penalty parameter, 
ill-conditioning and, ultimately, the divergence of the method. Therefore, we propose to use a different penalty
updating rule, which showed itself best in numerical experiments. It should be noted that this rule is completely
heuristic and its theoretical analysis is a challenging open problem.

We used the following two inequalities as a criterion for updating the penalty parameter:
\begin{equation} \label{eq:PenaltyUpdate_Heuristic}
  \| S_i x_i^n - y_i^n - c_i \| \ge \varkappa, \enspace
  \| S_i x_i^{n + 1} - y_i^{n + 1} - c_i \| > \eta \| S_i x_i^n - y_i^n - c_i \|, \enspace i \in \{ 1, 2 \},
\end{equation}
where $\varkappa > 0$ and $\eta \in (0, 1)$ are fixed parameters. If inequalities \eqref{eq:PenaltyUpdate_Heuristic} are
satisfied, then $\tau_{n + 1} = \beta \tau_n$ for some $\beta > 1$. Otherwise, we put $\tau_{n + 1} = \tau_n$. This way
the penalty parameter is \textit{not} updated, if the infeasibility measure $\| S_i x_i^n - y_i^n - c_i \|$, 
$i \in \{ 1, 2 \}$, is sufficiently small or decreases with linear rate.

\begin{remark}
Let us explain the motivation behind criterion \eqref{eq:PenaltyUpdate_Heuristic} for increasing the penalty parameter.
Multiple numerical experiments with fixed penalty parameter $\tau$ demonstrated that one needs to increase $\tau$ only
if the infeasibility measure $\| S_i x_i^n - y_i^n - c_i \|$, $i \in \{ 1, 2 \}$, does not decrease with iterations.
Therefore, it is natural to use the second inequality in \eqref{eq:PenaltyUpdate_Heuristic} as a criterion for updating
$\tau$. In addition, numerical experiments showed that at the first stage the ADMM converges to a nearly feasible
point and only later the sequence starts to converge to a point satisfying KKT optimality conditions. As a result, at
later stages, when the value $\| S_i x_i^n - y_i^n - c_i \|$ is very small but nonzero, the second inequality in
\eqref{eq:PenaltyUpdate_Heuristic} might be violated, although there is no need to increase $\tau$ at this stage.
Therefore, the first inequality in \eqref{eq:PenaltyUpdate_Heuristic} is used as a safeguard to avoid an unnecessary
increase of the penalty parameter. A safe choice of parameter $\varkappa$ is $\varkappa < \varepsilon$, where
$\varepsilon > 0$ is from the stopping criterion discussed below, since this choice makes penalty updating criterion 
\eqref{eq:PenaltyUpdate_Heuristic} consistent with the stopping criterion and ensures that the penalty parameter is
increased until the infeasibility measure is within the limits specified by the stopping criterion. Nevertheless, in
our experiments we defined $\varkappa = 0.1$, since this value ensured convergence to a KKT point for all test
problems and allowed one to avoid unnecessary penalty updates, which might slow down the convergence.
\end{remark}

It remains to define a stopping criterion. To this end, let us consider optimality conditions as in the convex
case. The KKT conditions for problem \eqref{prob:EllipsoidsDistNonconvex_ADMM} have the form:
\begin{gather} \label{eq:OptmCond_Nonconv_x}
  x_1^* - x_2^* - S_1 \lambda_1^* = 0, \quad
  x_2^* - x_1^* - S_2 \lambda_2^* = 0, \\
  \lambda_i^* + \mu_i^* y_i^* = 0, \quad S_i x_i^* - y_i^* = c_i, \quad \| y_i^* \| = 1 \quad i \in \{ 1, 2 \},
  \label{eq:OptmCond_Nonconv_y}
\end{gather}
where $\mu_i^* \in \mathbb{R}$ is a Lagrange multiplier corresponding to the equality constraint $\| y_i \| - 1 = 0$. 
Note that the first condition in \eqref{eq:OptmCond_Nonconv_y} is satisfied iff 
$\lambda_i^* = \pm \| \lambda_i^* \| y_i^*$.

The violation of optimality conditions \eqref{eq:OptmCond_Nonconv_x} and \eqref{eq:OptmCond_Nonconv_y} can be measured
with the use of the following functions:
\begin{gather*}
  R_x(x, y, \lambda) = \begin{pmatrix} x_1 - x_2 - S_1 \lambda_1 \\ x_2 - x_1 - S_2 \lambda_2 \end{pmatrix}, \quad
  R_{y_i}^{\pm}(x, y, \lambda) = \Big( \lambda_i \pm \| \lambda_i \| y_i \Big), \quad i \in \{ 1, 2 \},
  \\
  R_c(x, y, \lambda) = \begin{pmatrix} S_1 x_1 - y_1 - c_1 \\ S_2 x_2 - y_2 - c_2 \end{pmatrix}.
\end{gather*}
It is natural to use the condition
\begin{multline} \label{eq:StoppingCriterion_Nonconvex}
  \| R_x(x^{n + 1}, y^{n + 1}, \lambda^{n + 1}) \|
  + \sum_{i = 1}^2 \min\Big\{ \big\| R_{y_i}^-(x^{n + 1}, y^{n + 1}, \lambda^{n + 1}) \big\|,
  \big\| R_{y_i}^+(x^{n + 1}, y^{n + 1}, \lambda^{n + 1}) \big\| \Big\} 
  \\
  + \| R_c(x^{n + 1}, y^{n + 1}, \lambda^{n + 1}) \| < \varepsilon
\end{multline}
as a stopping criterion for the method.

Thus, one can propose the following scheme of the ADMM (given in Algorithm~\ref{alg:ADMM_nonconvex}) for solving
nonconvex problem \eqref{prob:EllipsoidsDistNonconvex_ADMM}, which is equivalent to the original problem 
\eqref{prob:EllipsoidDist_NonConvex} of finding the distance between the boundaries of two ellipsoids. 

\begin{algorithm}[t]	\label{alg:ADMM_nonconvex}
\caption{The ADMM for finding the distance between the boundaries of ellipsoids.}

\noindent\textbf{Step~1.} {Choose $y^0, \lambda^0 \in \mathbb{R}^{2d}$, $\tau_0 > 0$, $\eta \in (0, 1)$, $\beta > 1$,
$\varkappa > 0$, and $\varepsilon > 0$. Set $n := 0$ and \textbf{flag} ${}={}$ \textbf{true}.}

\noindent\textbf{Step~2.} {Compute the square roots $S_i$ of the matrices $Q_i$ and vectors $c_i = S_i z_i$, 
$i \in \{1, 2 \}$.}

\noindent\textbf{Step~3.} {If \textbf{flag} ${}={}$ \textbf{true}, compute the Cholesky decomposition 
$H(\tau_n) = L L^T$ of the matrix $H(\tau_n)$ given in \eqref{eq:QuadFormHessian} and put 
\textbf{flag} ${}={}$ \textbf{false}.}

\noindent\textbf{Step~4.} {Compute 
$u^n = \left(\begin{smallmatrix} S_1(\lambda_1^n + \tau_n(y_1^n + c_1)) \\ S_2(\lambda_2^n + \tau_n(y_1^n + c_1))
\end{smallmatrix}\right)$ and solve the linear systems $L w^n = u^n$ and $L^T x^{n + 1} = w^n$.
}

\noindent\textbf{Step~5.} {Compute $v_i^n = S_i x_i^{n + 1} - c_i - (1 / \tau_n) \lambda_i^n$, $i \in \{ 1, 2 \}$. If
$v_i^n \ne 0$, put $y_i^{n + 1} = \frac{1}{\| v_i^n \|} v_i^n$; otherwise, choose $z_i \in \mathbb{R}^d$ with 
$\| z_i \| = 1$ and define $y_i^{n + 1} = z_i$, $i \in \{1, 2 \}$.
}

\noindent\textbf{Step~6.} {Put $\lambda_i^{n + 1} = \lambda_i^n - \tau_n(S_i x_i^{n + 1} - y_i^{n + 1} - c_i)$, 
$i \in \{ 1, 2 \}$.
}

\noindent\textbf{Step~7.} {Compute $R_x(x^{n + 1}, y^{n + 1}, \lambda^{n + 1})$, 
$R_{y_i}^{\pm}(x^{n + 1}, y^{n + 1}, \lambda^{n + 1})$, and $R_c(x^{n + 1}, y^{n + 1}, \lambda^{n + 1})$. 
If condition \eqref{eq:StoppingCriterion_Nonconvex} holds true, then \textbf{Stop}. 
}

\noindent\textbf{Step~8.} {If $n \ge 1$, define
$$	
  \tau_{n + 1} = \begin{cases}
    \beta \tau_n, & \text{if inequalities \eqref{eq:PenaltyUpdate_Heuristic} hold true,} 
    \\
    \tau_n, & \text{otherwise}
  \end{cases}
$$ 
If $\tau_{n + 1} \ne \tau_n$, set \textbf{flag} ${}={}$ \textbf{true}. Put $n := n + 1$ and go to \textbf{Step 3}.
}
\end{algorithm}

As in Algorithm~\ref{alg:ADMM_sa_convex}, we used the boolean variable ``flag'' in order to indicate whether
the Cholesky decomposition must be updated. This way the decomposition $H(\tau_n) = L L^T$ is recomputed only when 
the penalty parameter $\tau_n$ has been updated (i.e. increased). Let us also note that the situation when 
$v_i^n = 0$ on Step~5 of Algorithm~\ref{alg:ADMM_nonconvex} never occurred in our numerical experiments, provided 
the initial guess $y^0$ satisfied the constraints $\| y^0_1 \| = \| y^0_2 \| = 1$. Nevertheless, to avoid potentially
incorrect behaviour of the algorithm one must include the case $v_i^n = 0$ into Step~5.

Finally, observe that on Step~8 of Algorithm~\ref{alg:ADMM_nonconvex} one can check the validity of the
inequality
$$
  \big\| R_c(x^{n + 1}, y^{n + 1}, \lambda^{n + 1}) \big\| \le \eta 
  \big\| R_c(x^n, y^n, \lambda^n) \big\|
$$
instead of verifying two similar inequalities involving $\| S_i x_i^n - y_i^n - c_i \|$.

\begin{remark}
Let us point out that various existing versions of the ADMM for nonconvex problems (e.g. the proximal ADMM from
\cite{BotNguyen}) can be directly applied to problem \eqref{prob:EllipsoidsDistNonconvex_ADMM} instead of
Algorithm~\ref{alg:ADMM_nonconvex}. A comparative analysis of different versions of the ADMM for solving problem
\eqref{prob:EllipsoidsDistNonconvex_ADMM} is an interesting problem for future research.
\end{remark}

\subsection{Analysis of the method}
\label{subsect:Theoretical_Analysis}

Although convergence analysis of the ADMM for nonconvex problems has recently become an active area of research (see 
\cite{YangPongChen2017,HajinezhadShi,HongLuoRazaviyayn,BotNguyen,GuoHanWu2017,MagnussonRabbat2016,ZhangLuo2020,
PengChenZhu2015,WangYinZeng,ThemelisPatrinos} and the references therein), to the best of the author's knowledge no
existing results on convergence of various modifications of this method are applicable to
Algorithm~\ref{alg:ADMM_nonconvex}. One of the main differences between our algorithm and other existing versions of 
the ADMM for nonconvex problems is the fact that the penalty parameter in the ADMM is usually assumed to be constant
throughout iterations (cf. 
\cite{YangPongChen2017,HajinezhadShi,WangYinZeng,ThemelisPatrinos,BotNguyen,GuoHanWu2017,MagnussonRabbat2016,
ZhangLuo2020}) or to increase unboundedly \cite{MagnussonRabbat2016}, while in Algorithm~\ref{alg:ADMM_nonconvex} we
update the penalty parameter $\tau_n$ adaptively. It should be noted that although adaptive penalty updates can improve
overall performance of the ADMM, they significantly complicate convergence analysis even in the convex case, as is
pointed out in \cite{EcksteinYao}.

Since inequalities \eqref{eq:PenaltyUpdate_Heuristic} were chosen heuristically without any theoretical
foundation, we will analyse the method under the assumption that inequality \eqref{eq:PenaltyUpdate_Nonconvex} is
used as a criterion for penalty updates. Although, this criterion performed poorly in our numerical experiments, its
analysis provides an insight into the performance of the ADDM in the nonconvex case and the choice of parameters of this
method.

We will analyse convergence in two different cases: when the sequence of penalty parameters $\{ \tau_n \}$ is bounded
(i.e. when Algorithm~\ref{alg:ADMM_nonconvex} updates the penalty parameter $\tau_n$ only a finite number of times) and
when this sequence is unbounded. In the first case the method converges to a point satisfying KKT optimality conditions
for problem \eqref{prob:EllipsoidsDistNonconvex_ADMM} with linear rate. In the second case the analysis of convergence
is much more complicated and we provide only a partial result on the convergence of the method.

\begin{theorem}
Let $\{ x^n \}$, $\{ y^n \}$, and $\{ \lambda^n \}$ be the sequences generated by Algorithm~\ref{alg:ADMM_nonconvex}
with inequalities \eqref{eq:PenaltyUpdate_Heuristic} on Step~8 replaced by inequality
\eqref{eq:PenaltyUpdate_Nonconvex}. Suppose also that the sequence of penalty parameters $\{ \tau_n \}$ is bounded. Then
the sequence $\{ (x^n, y^n, \lambda^n) \}$ converges to a point $(x^*, y^*, \lambda^*)$ satisfying KKT optimality
conditions for problem \eqref{prob:EllipsoidsDistNonconvex_ADMM} and there exists $M > 0$ such that
\begin{equation} \label{eq:ConvergenceLinearRate}
  \| x^* - x^n \| \le M \eta^n, \quad \| y^* - y^n \| \le M \eta^n, \quad
  \| \lambda^* - \lambda^n \| \le M \eta^n \quad \forall n \in \mathbb{N}.
\end{equation}
\end{theorem}

\begin{proof}
From the fact that the sequence $\{ \tau_n \}$ is bounded it follows that there exists $n_0 \in \mathbb{N}$ such that
for all $n \ge n_0$ inequality \eqref{eq:PenaltyUpdate_Nonconvex} is not satisfied, i.e. for all $n \ge n_0$ one has 
\begin{equation} \label{eq:BoundedPenaltyParamCond}
\begin{split}
  \| S_i x_i^{n + 1} - y_i^n - c_i \| &\le \eta \| S_i x_i^n - y_i^{n - 1} - c_i \|, 
  \\
  \| S_i x_i^{n + 1} - y_i^{n + 1} - c_i \| &\le \eta \| S_i x_i^n - y_i^n - c_i \|
\end{split}
\end{equation}
(see Step~8 of Algorithm~\ref{alg:ADMM_nonconvex}). Therefore, according to Step~6 for any $i \in \{ 1, 2 \}$ and 
$n \ge n_0$ one has
$$
  \| \lambda_i^{n + 1} - \lambda_i^n \| = 
  \tau_n \big\| S_i x_i^{n + 1} - y_i^{n + 1} - c_i \big\| 
  \le \big( \sup_{n \in \mathbb{N}} \tau_n \big) \big\| S_i x_i^{n_0} - y_i^{n_0} - c_i \big\| \eta^{n - n_0}.
$$
Consequently, for any $m > n \ge n_0$ one has
$$
  \| \lambda_i^m - \lambda_i^n \| \le \sum_{k = n}^{m - 1} \| \lambda_i^{k + 1} - \lambda_i^k \|
  \le \big( \sup_{n \in \mathbb{N}} \tau_n \big) \big\| S_i x_i^{n_0} - y_i^{n_0} - c_i \big\| \eta^{n - n_0} 
  \frac{1 - \eta^{m - n}}{1 - \eta}.
$$
Hence bearing in mind the fact that $\eta \in (0, 1)$ one obtains that $\{ \lambda^n \}$ is a Cauchy sequence,
which implies that it converges to some $\lambda^* \in \mathbb{R}^{2d}$. Moreover, for any $n \ge n_0$ one has
\begin{multline*}
  \| \lambda_i^* - \lambda_i^n \| = \Big\| \sum_{k = n}^{\infty} (\lambda_i^{k + 1} - \lambda_i^k) \Big\|
  \le \sum_{k = n}^{\infty} \big( \sup_{n \in \mathbb{N}} \tau_n \big) 
  \big\| S_i x_i^{n_0} - y_i^{n_0} - c_i \big\| \eta^{k - n_0} = M_{\lambda} \eta^n,  
\end{multline*}
where
$$
  M_{\lambda} = \big( \sup_{n \in \mathbb{N}} \tau_n \big) \big\| S_i x_i^{n_0} - y_i^{n_0} - c_i \big\| 
  \frac{1}{(1 - \eta) \eta^{n_0}}.
$$
Thus, there exists $M > 0$ such that the third inequality in \eqref{eq:ConvergenceLinearRate} holds true for
all $n \in \mathbb{N}$.

Fix $i \in \{ 1, 2 \}$. Recall that for all $n \ge n_0$ inequality \eqref{eq:BoundedPenaltyParamCond} holds true.
Therefore 
\begin{equation} \label{eq:Inter_Iteration_Residual}
  S_i x_i^{n + 1} - y_i^n - c_i = \rho_i^n, \quad 
  \| \rho_i^n \| \le \big\| S_i x_i^{n_0} - y_i^{n_0 - 1} - c_i \big\| \eta^{n - n_0} \quad \forall n \ge n_0.
\end{equation}
Taking into account the definition of $\lambda_i^{n + 1}$ (see Step~6 of Algorithm~\ref{alg:ADMM_nonconvex}) one
obtains that
$$
  y_i^{n + 1} = \frac{1}{\tau_n} \big( \lambda_i^{n + 1} - \lambda_i^n \big) + S_i x_i^{n + 1} - c_i
  = \frac{1}{\tau_n} \big( \lambda_i^{n + 1} - \lambda_i^n \big) + y_i^n + \rho_i^n.
$$
Hence with the use of the third inequality in \eqref{eq:ConvergenceLinearRate} one gets that for any $n \ge n_0$ the
following inequalities hold true:
$$
  \big\| y_i^{n + 1} - y_i^n \big\| \le \frac{1}{\tau_n} \big\| \lambda_i^{n + 1} - \lambda_i^n \big\| + \| \rho_i^n \| 
  \le \frac{M}{\tau_0} \eta^n + \big\| S_i x_i^{n_0} - y_i^{n_0 - 1} - c_i \big\| \eta^{n - n_0}.
$$
Now, arguing in the same way as in the case of the sequence $\{ \lambda^n \}$ one can check that $\{ y_i^n \}$ is a
Cauchy sequence, which implies that it converges to some $y_i^*$, and the second inequality in
\eqref{eq:ConvergenceLinearRate} is satisfied for all $n \in \mathbb{N}$.

As was noted above, under the assumptions of the theorem inequality \eqref{eq:BoundedPenaltyParamCond} is satisfied
for all $n \ge n_0$. Therefore
$$
  S_i x_i^n - y_i^n - c_i = \xi_i^n, \quad 
  \| \xi_i^n \| \le \big\| S_i x_i^{n_0} - y_i^{n_0} - c_i \big\| \eta^{n - n_0} 
  \quad \forall n \ge n_0, \: i \in \{ 1, 2 \}.
$$
Hence with the use of the facts that $y^n \to y^*$ as $n \to \infty$ and $\eta \in (0, 1)$ one obtains that the sequence
$\{ x_i^n \}$ converges to the point $x_i^* = S_i^{-1} (y_i^* + c_i)$, that is, $S_i x_i^* - y_i^* = c_i$. Furthermore,
observe that $\| y_i^* \| = 1$, since $\| y_i^n \| = 1$ for all $n \in \mathbb{N}$ 
(see Step~5 of Algorithm~\ref{alg:ADMM_nonconvex}). Thus, $(x^*, y^*)$ is a feasible point of problem
\eqref{prob:EllipsoidsDistNonconvex_ADMM}. In addition, one has
$$
  \| x_i^* - x_i^n \| = \| S_i^{-1} \| \big\| y_i^* - y_i^n - \xi_i^n \big\| \le
  \| S_i^{-1} \| \big\| y_i^* - y_i^n \big\| 
  + \| S_i^{-1} \| \big\| S_i x_i^{n_0} - y_i^{n_0} - c_i \big\| \eta^{n - n_0}
$$
for all $n \ge n_0$. Therefore, the first inequality in \eqref{eq:ConvergenceLinearRate} follows directly from
the second one. 

Let us finally check that the triplet $(x^*, y^*, \lambda^*)$ satisfies KKT optimality conditions for problem
\eqref{prob:EllipsoidsDistNonconvex_ADMM} (see \eqref{eq:OptmCond_Nonconv_x}, \eqref{eq:OptmCond_Nonconv_y}). Indeed,
as was noted above, one has $S_i x_i^* - y_i^* = c_i$ and $\| y_i^* \| = 1$, $i \in \{ 1, 2 \}$, that is, $(x^*, y^*)$
is a feasible point of problem \eqref{prob:EllipsoidsDistNonconvex_ADMM}.

By definition the point $x^{n + 1}$ is the solution of the following system of equations
$$
  \begin{pmatrix} I_d + \tau_n Q_1 & - I_d \\ - I_d & I_d + \tau_n Q_2 \end{pmatrix} x^{n + 1} =
  \begin{pmatrix} S_1 & \mathbb{O}_d \\ \mathbb{O}_d & S_2 \end{pmatrix} \Big( \lambda^n + \tau_n( y^n + c ) \Big).
$$
(see Steps 3 and 4 of Algorithm~\ref{alg:ADMM_nonconvex} and equality \eqref{eq:LinearSystem_ConvexCase}). Therefore,
for any $n \in \mathbb{N}$ one has
\begin{align*}
  x_1^{n + 1} - x_2^{n + 1} - S_1 \lambda_1^n &= \tau_n S_1 \Big( - S_1 x_1^{n + 1} + y_1^n + c_1 \Big), \\
  x_2^{n + 1} - x_1^{n + 1} - S_2 \lambda_2^n &= \tau_n S_2 \Big( - S_2 x_2^{n + 1} + y_2^n + c_2 \Big)
\end{align*}
(here we used the equality $Q_i = S_i S_i$). Hence adding and subtracting $y_i^{n + 1}$ and taking into account the 
definition of $\lambda_i^{n + 1}$ (see Step~6) one gets that
\begin{align*}
  x_1^{n + 1} - x_2^{n + 1} - S_1 \lambda_1^n &= S_1 \big( \lambda_1^{n + 1} - \lambda_1^n \big)
  + \tau_n S_1 \big( y_1^n - y_1^{n + 1} \big), 
  \\
  x_2^{n + 1} - x_1^{n + 1} - S_2 \lambda_2^n &= S_2 \big( \lambda_2^{n + 1} - \lambda_2^n \big)
  + \tau_n S_2 \big( y_2^n - y_2^{n + 1} \big).
\end{align*}
Passing to the limit as $n \to \infty$ with the use of the fact that the sequence $\tau_n$ is bounded one obtains
\begin{equation} \label{eq:KKT_BoundedPenalty}
  x_1^* - x_2^* - S_1 \lambda_1^* = 0, \quad x_2^* - x_1^* - S_2 \lambda_2^* = 0, 
\end{equation}
i.e. optimality condition \eqref{eq:OptmCond_Nonconv_x} holds true.

Fix $i \in \{ 1, 2 \}$. If $\lambda_i^* = 0$, then $\lambda_i^* + 0 y_i^* = 0$. Therefore, suppose that 
$\lambda_i^* \ne 0$. Let us consider two cases. Suppose at first that there exists $n^* \in \mathbb{N}$ such that 
$v_i^n \ne 0$ for all $n \ge n^*$ (see Step~5 of Algorithm~\ref{alg:ADMM_nonconvex}). Then with the use of the
definitions of $y_i^{n + 1}$ and $\lambda_i^{n + 1}$ (Steps~5 and 6 of Algorithm~\ref{alg:ADMM_nonconvex}) one obtains
that
$$
  y_i^* = \lim_{n \to \infty} y_i^{n + 1} = 
  \lim_{n \to \infty}
  \frac{y_i^{n + 1} - \frac{1}{\tau_n} \lambda_i^{n + 1}}{\| y_i^{n + 1} - \frac{1}{\tau_n} \lambda_i^{n + 1} \|}
  = \frac{y_i^* - \frac{1}{\tau_{n_0}} \lambda_i^*}{\| y_i^* - \frac{1}{\tau_{n_0}} \lambda_i^* \|}.
$$
Consequently, $\lambda_i^* + \mu_i^* y_i^* = 0$, where $\mu_i^* = \tau_{n_0} - \| \tau_{n_0} y_i^* - \lambda_i^* \|$.
Thus, the triplet $(x^*, y^*, \lambda^*)$ satisfies KKT optimality conditions \eqref{eq:OptmCond_Nonconv_x},
\eqref{eq:OptmCond_Nonconv_y} for problem \eqref{prob:EllipsoidsDistNonconvex_ADMM} .

Suppose now that there exists a subsequence $\{ v_i^{n_k} \}$ such that $v_i^{n_k} = 0$ for all $k \in \mathbb{N}$.
Then by definition (see Step~5 of Algorithm~\ref{alg:ADMM_nonconvex}) one has
\[
  S_i x_i^{n_k + 1} - c_i = \frac{1}{\tau_{n_k}} \lambda_i^{n_k} = y_i^{n_k} + \rho_i^{n_k} 
  \quad \forall k \in \mathbb{N},
\]
where $\rho_i^n = S_i x_i^{n + 1} - y_i^n - c_i$. As was noted above, $\rho_i^n \to 0$ as $n \to \infty$. Therefore,
passing to the limit as $k \to \infty$ one obtains that $\lambda_i^* = \tau_{n_0} y_i^*$, which completes the proof.
\end{proof}

\begin{remark}
Note that the parameter $\eta \in (0, 1)$ from criterion \eqref{eq:PenaltyUpdate_Nonconvex} for penalty updates is used
in the upper estimates \eqref{eq:ConvergenceLinearRate} of the rate of convergence of the ADMM. Therefore $\eta$ must be
greater than the actual rate of convergence of the method, i.e.
$$
  \max\left\{ \limsup_{n \to \infty} \frac{\| x^* - x^{n + 1} \|}{\| x^* - x^n \|},
  \limsup_{n \to \infty} \frac{\| y^* - y^{n + 1} \|}{\| y^* - y^n \|}, 
  \limsup_{n \to \infty} \frac{\| \lambda^* - \lambda^{n + 1} \|}{\| \lambda^* - \lambda^n \|} \right\} < \eta,
$$
since otherwise the method would start increasing the penalty parameter $\tau_n$, when it is absolutely unnecessary. In
particular, it seems advisable to choose $\eta$ (in both \eqref{eq:PenaltyUpdate_Nonconvex} and
\eqref{eq:PenaltyUpdate_Heuristic}) to be sufficiently close to $1$ in order to avoid an unbounded increase
of the penalty parameter $\tau_n$ for ill-conditioned problems.
\end{remark}

Let us now consider the case when the penalty parameter $\tau_n$ increases unboundedly with iterations. In this case a
convergence analysis of the method is much more complicated and we present only an incomplete result, which can be
viewed as an intermediate lemma in a comprehensive convergence analysis of the ADMM for finding the distance between 
the boundaries of ellipsoids. Nevertheless, this result highlights some peculiarities of the method in the case when
the penalty parameter increases unboundedly, and we hope that it might help the interested reader to develop a more
complete convergence theory for the ADMM in the nonconvex case.

Let us first prove an auxiliary lemma, which, in particular, implies that for any sufficiently large value of the
penalty parameter $v_i^n \ne 0$ on Step~5 of Algorithm~\ref{alg:ADMM_nonconvex}, provided multipliers $\{ \lambda^n \}$
lie within a bounded set. 

\begin{lemma} \label{lem:ADMM_Nondegeneracy}
Let $\Lambda \subset \mathbb{R}^{2d}$ be a bounded set. Then there exists $\tau^* > 0$ such that for any 
$\tau \ge \tau^*$ and for all $\lambda \in \Lambda$ and 
$y = \left( \begin{smallmatrix} y_1 \\ y_2 \end{smallmatrix} \right) \in \mathbb{R}^{2d}$ with 
$\| y_1 \| = \| y_2 \| = 1$ solution $x^*$ of the system of linear equations
\[
  \begin{pmatrix} I_d + \tau Q_1 & - I_d \\ - I_d & I_d + \tau Q_2 \end{pmatrix} x^* =
  \begin{pmatrix} S_1 & \mathbb{O}_d \\ \mathbb{O}_d & S_2 \end{pmatrix} \Big( \lambda + \tau( y + c ) \Big).
\]
satisfies the conditions $S_i x_i^* - c_i - (1/\tau) \lambda_i \ne 0$, $i \in \{ 1, 2 \}$.
\end{lemma}

\begin{proof}
Let us verify that the statement of the lemma holds true for any $\tau^* > \max\{ 1, K \}$, where
\[
  K =
  \Big( \sum_{i = 1}^2 \| S_i^{-1} \| \big( 1 + \| c_i \| + \sup_{\lambda \in \Lambda} \| \lambda_i \| \big) \Big)^2.
\]
Indeed, fix any $\tau \ge \tau^* > \max\{ 1, K \}$, $\lambda \in \Lambda$, and $y \in \mathbb{R}^{2d}$ with 
$\| y_1 \| = \| y_2 \| = 1$. Let $x^*$ be a solution of the corresponding system of linear equations. 

As is easily seen, $x^*$ is a point of global minimum of the convex function
\[
  g(x) = \frac{1}{2} \| x_1 - x_2 \|^2 
  + \sum_{i = 1}^2 \frac{\tau}{2} \Big\| S_i x_i - y_i - c_i - \frac{1}{\tau} \lambda_i \Big\|^2,
\]
since it satisfies the equality $\nabla g(x^*) = 0$. Consequently, $g(x^*) \le g(\widehat{x})$, where
\[
  \widehat{x}_i = S_i^{-1} \Big( y_i + c_i + \frac{1}{\tau} \lambda_i \Big), \quad i \in \{ 1, 2 \}.
\]
Taking into account the fact that $\tau \ge \tau^* > 1$ one obtains that
\[
  \| \widehat{x}_1 - \widehat{x}_2 \| \le \| \widehat{x}_1 \| + \| \widehat{x}_2 \|
  \le \sum_{i = 1}^2 \| S_i^{-1} \| \Big( 1 + \| c_i \| + \frac{1}{\tau} \| \lambda_i \| \Big) \le \sqrt{K}.
\]
Therefore
\[
  g(x^*) \le g(\widehat{x}) = \frac{1}{2} \| \widehat{x}_1 - \widehat{x}_2 \|^2 \le \frac{K}{2} 
  < \frac{\tau_*}{2} \le \frac{\tau}{2}.
\]
On the other hand, if $S_i x_i^* - c_i - (1/\tau) \lambda_i = 0$ for some $i \in \{ 1, 2 \}$, then
\[
  g(x^*) \ge \frac{\tau}{2} \| y_i \|^2 = \frac{\tau}{2} > g(x^*),
\]
which is impossible. Thus, $S_i x_i^* - c_i - (1/\tau) \lambda_i \ne 0$, $i \in \{ 1, 2 \}$.
\end{proof}

Next we present a partial result on convergence of Algorithm~\ref{alg:ADMM_nonconvex} in the case when the penalty
parameter $\{ \tau_n \}$ increases unboundedly as $n \to \infty$.

\begin{proposition}
Let $\{ x^n \}$, $\{ y^n \}$, and $\{ \lambda^n \}$ be the sequences generated by Algorithm~\ref{alg:ADMM_nonconvex} 
with inequalities \eqref{eq:PenaltyUpdate_Heuristic} on Step~8 replaced by inequality
\eqref{eq:PenaltyUpdate_Nonconvex}. Suppose that the sequence of penalty parameters $\{ \tau_n \}$ is unbounded, but 
the sequence of multipliers $\{ \lambda^n \}$ is bounded. Then the sequence $\{ (x^n, y^n) \}$ is bounded as well and
all its limit points are feasible for problem \eqref{prob:EllipsoidsDistNonconvex_ADMM}. Moreover, all limit points of
the sequence $\{ (x^n, y^n, \lambda^n) \}$ satisfy KKT optimality conditions for problem
\eqref{prob:EllipsoidsDistNonconvex_ADMM} if and only if $\tau_n \| y^{n + 1} - y^n \| \to 0$ as $n \to \infty$.
\end{proposition}

\begin{proof}
For any $n \in \mathbb{N}$ denote $\xi_i^n = S_i x_i^{n + 1} - y_i^n - c_i$, $i \in \{ 1, 2 \}$. Let us show that
$\xi_i^n \to 0$ as $n \to \infty$. Indeed, denote $\widehat{x}_i^n = S_i^{-1}(y_i^n + c_i + (1 / \tau_n) \lambda_i^n)$
for all $n \in \mathbb{N}$ and $i \in \{ 1, 2 \}$. The sequence $\{ \widehat{x}^n \}$ is bounded due to the boundedness
of the sequence $\{ \lambda^n \}$ and the fact that $\| y_i^n \| = 1$ for all $n \in \mathbb{N}$ (see Step~5 of
Algorithm~\ref{alg:ADMM_nonconvex}). 

Recall that by definition $x^{n + 1}$ is a point of global minimum of the quadratic function $f_n(x)$ defined in
\eqref{eq:QuadForm} with $\tau = \tau_n$. Therefore $x^{n + 1}$ is also a point of global minimum of the function
\begin{align*}
  g_n(x) &= f_n(x) 
  + \sum_{i = 1}^2 \Big( \langle \lambda_i^n, y_i^n + c_i \rangle + \frac{1}{2\tau_n} \| \lambda_i^n \|^2 \Big)
  \\
  &= \frac{1}{2} \| x_1 - x_2 \|^2 
  + \sum_{i = 1}^2 \frac{\tau_n}{2} \Big\| S_i x_i - y_i^n - c_i - \frac{1}{\tau_n} \lambda_i^n \Big\|^2,
\end{align*}
which implies that $g_n(x^{n + 1}) \le g_n(\widehat{x}^n)$ for all $n \in \mathbb{N}$. Hence with the use of the
equality $g_n(\widehat{x}^n) = 0.5 \| \widehat{x}_1^{n} - \widehat{x}_2^n \|^2$ and the boundedness of the sequence 
$\{ \widehat{x}^n \}$ one obtains that the sequence $\{ g_n(x^{n + 1}) \}$ is bounded above. 

Arguing by reductio ad absurdum, suppose that the sequence $\{ \xi_1^n \}$ does not converge to zero (the convergence
of the sequence $\{ \xi_2^n \}$ to zero is proved in the same way). Then there exist $\delta > 0$ and a subsequence 
$\{ \xi_1^{n_k} \}$ such that $\| \xi_1^{n_k} \| \ge \delta$ for all $k \in \mathbb{N}$. Observe that for any
$k \in \mathbb{N}$ one has
\begin{align*}
  g_{n_k}(x^{n_k + 1}) &\ge \frac{1}{2} \| x_1^{n_k + 1} - x_2^{n_k + 1} \|^2 + 
  \sum_{i = 1}^2 \Big( \frac{\tau_{n_k}}{2} \| \xi_i^{n_k} \|^2 - \| \lambda_i^{n_k} \|  \| \xi_i^{n_k} \| 
  + \frac{1}{2\tau_{n_k}} \| \lambda_i^{n_k} \|^2 \Big)
  \\
  &\ge 
  \frac{1}{2} \Big(\sqrt{\tau_{n_k}} \| \xi_1^{n_k} \| - \frac{1}{\sqrt{\tau_{n_k}}} \| \lambda_1^{n_k} \| \Big)^2. 
\end{align*}
Hence bearing in mind the facts that by our assumptions the sequence $\{ \lambda^n \}$ is bounded, while the sequence
$\{ \tau_n \}$ increases unboundedly one can readily verify that $g_{n_k}(x^{n_k + 1}) \to + \infty$ as $k \to \infty$,
which contradicts the fact that the sequence $\{ g_n(x^{n + 1}) \}$ is bounded above. 

Thus, $\xi_i^n \to 0$ as $n \to \infty$, $i \in \{ 1, 2 \}$. Recall that $\| y_i^n \| = 1$ for all $n \in \mathbb{N}$
(see Step~5 of Algorithm~\ref{alg:ADMM_nonconvex}). Therefore
$$
  \| x_i^{n + 1} \| = \big\| S_i^{-1} (y_i^n + c_i + \xi_i^n) \| \le 
  \| S_i^{-1} \| \big( 1 + \| c_i \| + \| \xi_i^n \| \big),
$$
which implies that the sequence $\{ (x^n, y^n) \}$ is bounded. Let $(x^*, y^*)$ be a limit point of this sequence, i.e.
there exists a subsequence $\{ (x^{n_k}, y^{n_k}) \}$ converging to $(x^*, y^*)$. Let us check that this point is 
feasible for problem \eqref{prob:EllipsoidsDistNonconvex_ADMM}.

Indeed, from the fact that $\| y_i^n \| = 1$ for all $n \in \mathbb{N}$ it follows that $\| y_i^* \| = 1$. Taking into
account the definition of $\lambda_i^{n + 1}$ (see Step~6 of Algorithm~\ref{alg:ADMM_nonconvex}) one obtains that
$$
  \big\| S_i x_i^{n + 1} - y_i^{n + 1} - c_i \big\| 
  \le \frac{2}{\tau_n} \sup_{n \in \mathbb{N} \cup \{ 0 \}} \| \lambda_i^n \|.
$$
Therefore, $\zeta_i^{n + 1} = S_i x_i^{n + 1} - y_i^{n + 1} - c_i \to 0$ as $n \to \infty$ due the boundedness of 
the sequence $\{ \lambda^n \}$ and the fact that $\tau_n \to + \infty$ as $n \to \infty$. Hence passing to the limit as
$k \to \infty$ in the equality
$$
  x_i^{n_k} = S_i^{-1} (y_i^{n_k} + c_i + \zeta_i^{n_k})
$$
one obtains that $S_i x_i^* - y_i^* = c_i$, that is, the point $(x^*, y^*)$ is feasible for problem
\eqref{prob:EllipsoidsDistNonconvex_ADMM}.

Suppose now that all limit points of the sequence $\{ (x^n, y^n, \lambda^n) \}$ satisfy KKT optimality conditions for
problem \eqref{prob:EllipsoidsDistNonconvex_ADMM}. Arguing by reductio ad absurdum, assume that the sequence 
$\tau_n \| y^{n + 1} - y^n \|$ does not converge to zero. Then there exist a subsequence of this sequence and 
$\varepsilon > 0$ such that $\tau_{n_k - 1} \| y^{n_k} - y^{n_k - 1} \| \ge \varepsilon$. The sequence
$\{ (x^{n_k}, y^{n_k}, \lambda^{n_k}) \}$ is bounded. Replacing this sequence, if necessary, with its subsequence one
can suppose that $\{ (x^{n_k}, y^{n_k}, \lambda^{n_k}) \}$ converges to some $(x^*, y^*, \lambda^*)$. By definition the
point $x^{n + 1}$ is the solution of the following system of equations
$$
  \begin{pmatrix} I_d + \tau_n Q_1 & - I_d \\ - I_d & I_d + \tau_n Q_2 \end{pmatrix} x^{n + 1} =
  \begin{pmatrix} S_1 & \mathbb{O}_d \\ \mathbb{O}_d & S_2 \end{pmatrix} \Big( \lambda^n + \tau_n( y^n + c ) \Big).
$$
(see Steps 3 and 4 of Algorithm~\ref{alg:ADMM_nonconvex} and equality \eqref{eq:QuadFormHessian}). Therefore, for any
$k \in \mathbb{N}$ one has
\begin{align*}
  x_1^{n_k} - x_2^{n_k} - S_1 \lambda_1^{n_k - 1} 
  &= \tau_{n_k - 1} S_1 \Big( - S_1 x_1^{n_k} + y_1^{n_k - 1} + c_1 \Big),  \\
  x_2^{n_k} - x_1^{n_k} - S_2 \lambda_2^{n_k - 1} 
  &= \tau_{n_k - 1} S_2 \Big( - S_2 x_2^{n_k} + y_2^{n_k - 1} + c_2 \Big)
\end{align*}
(here we used the equality $Q_i = S_i S_i$). Adding and subtracting $y_i^{n_k}$ and taking into account the 
definition of $\lambda_i^{n + 1}$ (see Step~6) one obtains that
\begin{equation} \label{eq:SeqOptimality_in_x}
\begin{split}
  x_1^{n_k} - x_2^{n_k} - S_1 \lambda_1^{n_k} 
  &= \tau_{n_k - 1} S_1 \Big( y_1^{n_k - 1} - y_1^{n_k} \Big),
  \\
  x_2^{n_k} - x_1^{n_k} - S_2 \lambda_2^{n_k} 
  &= \tau_{n_k - 1} S_2 \Big( y_2^{n_k - 1} - y_2^{n_k} \Big)
\end{split}
\qquad \forall k \in \mathbb{N}.
\end{equation}
Therefore for all $k \in \mathbb{N}$ one has
\begin{align*}
  \big\| x_1^{n_k} - x_2^{n_k} - S_1 \lambda_1^{n_k} \big\| 
  &\ge \frac{1}{\| S_1^{-1} \|} \tau_{n_k - 1} \big\| y_1^{n_k - 1} - y_1^{n_k} \big\| 
  \ge \frac{\varepsilon}{\| S_1^{-1} \|}, 
  \\
  \big\| x_2^{n_k} - x_1^{n_k} - S_2 \lambda_2^{n_k} \big\|
  &\ge \frac{1}{\| S_2^{-1} \|} \tau_{n_k - 1} \big\| y_2^{n_k - 1} - y_2^{n_k} \big\| 
  \ge \frac{\varepsilon}{\| S_2^{-1} \|},
\end{align*}
which contradicts the fact that by our assumption all limit points of the sequence $\{ (x^n, y^n, \lambda^n) \}$
(including $(x^*, y^*, \lambda^*)$) satisfy KKT optimality conditions for problem 
\eqref{prob:EllipsoidsDistNonconvex_ADMM} (see \eqref{eq:OptmCond_Nonconv_x}).

Suppose finally that $\tau_n \| y^{n + 1} - y^n \| \to 0$ as $n \to \infty$. Let $(x^*, y^*, \lambda^*)$ be a limit
point of the sequence $\{ (x^n, y^n, \lambda^n) \}$, i.e. there exists a subsequence 
$\{ (x^{n_k}, y^{n_k}, \lambda^{n_k}) \}$ converging to $(x^*, y^*, \lambda^*)$. Passing to the limit as $k \to \infty$
in \eqref{eq:SeqOptimality_in_x} one obtains that
$$
  x_1^* - x_2^* - S_1 \lambda_1^* = 0, \quad x_2^* - x_1^* - S_2 \lambda_2^* = 0.
$$
Recall that by our assumption the sequence $\{ \lambda_n \}$ is bounded and $\tau_n \to + \infty$ as $n \to \infty$.
Therefore by Lemma~\ref{lem:ADMM_Nondegeneracy} there exists $n_0 \in \mathbb{N}$ such that for all $n \ge n_0$ one has
$v_i^n \ne 0$, $i \in \{ 1, 2 \}$ (see Step 5 of Algorithm~\ref{alg:ADMM_nonconvex}). Hence by the definitions of 
$y^{n + 1}$ and $\lambda^{n + 1}$ (Steps 5 and 6 of Algorithm~\ref{alg:ADMM_nonconvex}) one has
$$
  y_i^{n + 1} = \frac{y_i^{n + 1} - \tau_n^{-1} \lambda_i^{n + 1}}{\| y_i^{n + 1} - \tau_n^{-1} \lambda_i^{n + 1} \|},
  \quad i \in \{ 1, 2 \} \quad \forall n \ge n_0.
$$
or, equivalently,
\begin{equation} \label{eq:SequentialOptim_lambda}
  \lambda_i^{n + 1} = \mu_i^{n + 1} y_i^{n + 1}, \quad 
  \mu_i^{n + 1} = \tau_n - \big\| \tau_n y_i^{n + 1} - \lambda_i^{n + 1} \big\|
  \quad \forall n \ge n_0.
\end{equation}
Recall that $\| y_i^n \| = 1$ for all $n \in \mathbb{N}$, $i \in \{ 1, 2 \}$. Therefore
$|\mu_i^{n + 1}| = \| \lambda_i^{n + 1} \|$ for all $n \ge n_0$ and the sequence $\{ \mu_i^n \}$ is bounded by
our assumption on the boundedness of multipliers $\lambda^n$. Consequently, there exists a subsequence of the sequence
$\{ \mu_i^{n_k} \}$, which we denote again by $\{ \mu_i^{n_k} \}$, converging to some $\mu_i^*$. With the use of
\eqref{eq:SequentialOptim_lambda} one obtains that $\lambda_i^* = \mu_i^* y_i^*$, which implies the required result (see
KKT optimality conditions \eqref{eq:OptmCond_Nonconv_x}, \eqref{eq:OptmCond_Nonconv_y} for problem
\eqref{prob:EllipsoidsDistNonconvex_ADMM}).
\end{proof}

\begin{remark}
The assumption on the boundedness of multipliers $\{ \lambda^n \}$ is standard in the theory of augmented Lagrangian
methods (cf.~\cite{BirginMartinez,LuoSunWu2008}). There are several well-known techniques guaranteeing the boundedness
of multipliers and convergence of augmented Lagrangian methods: safeguarding, conditional multiplier updating,
normalization of multipliers, etc. All these techniques were discussed in detail, e.g. in
\cite{LuoSunWu2008,LuoSunLi2007,WangLi2009}, and can be applied to Algorithm~\ref{alg:ADMM_nonconvex}. However, as is
pointed out in the monograph \cite{BirginMartinez}, in applications multipliers usually remain bounded without the use
of any special techniques ensuring their boundedness. Moreover, in all our numerical experiments on various types of
test problems multipliers $\{ \lambda^n \}$ always remained bounded. That is why we did not include any of 
the aforementioned techniques into the description of Algorithm~\ref{alg:ADMM_nonconvex}.
\end{remark}

\subsection{A heuristic restarting procedure and a global method}
\label{sect:HeuristicRestart}

The problem of finding the distance between the boundaries of ellipsoids is nonconvex. In particular, not all KKT points
of this problem are its globally optimal solutions. Consequently, a sequence generated by  
Algorithm~\ref{alg:ADMM_nonconvex} might converge to a locally optimal solution or even just a stationary point of
problem \eqref{prob:EllipsoidsDistNonconvex_ADMM}. To overcome this difficulty, we propose to use a simple heuristic
restarting procedure (a new starting point), which can be applied to any local method for solving problem
\eqref{prob:EllipsoidDist_NonConvex} (or problem \eqref{prob:EllipsoidsDistNonconvex_ADMM}).

Namely, if Algorithm~\ref{alg:ADMM_nonconvex} terminates by finding a point $(x^*, y^*, \lambda^*)$, then we propose to
restart the algorithm at a point $x^0 = (x^0_1, x^0_2)$ such that the points $x^0_i$ are, in a sense, diametrically
opposed to $x^*_i$ with respect to the centre of the corresponding ellipsoid. Namely, we define
$$
  x^0_i = (-x_i^* + z_i) + z_i = - x_i^* + 2 z_i, \quad i \in \{ 1, 2 \}.
$$
Bearing in mind the constraints of problem \eqref{prob:EllipsoidsDistNonconvex_ADMM} we put
$y_i^0 = S_i x_i^0 - c_i = - S_i x_i^* + c_i$, $i \in \{ 1, 2 \}$. The initial value of multipliers can be defined
arbitrarily.

Roughly speaking, we restart the algorithm at the opposite `sides' of the ellipsoids. In the case $z_1 = z_2 = 0$ one
simply sets $x^0 = - x^*$. To obtain the formula for $x^0$ in the general case one has to shift the origin to the
centre $z_i$ of the corresponding ellipsoid, take the vector opposite to $x_i^* - z_i$, and then shift the origin back
to its place.

Note that if the algorithm finds a point at which the ellipsoids intersect, there is no need to restart it at a new
initial point. Therefore, we arrive at the following restarting scheme for Algorithm~\ref{alg:ADMM_nonconvex} given in
Algorithm~\ref{alg:ADMM_nonconvex_global}.

\begin{algorithm}[t]	\label{alg:ADMM_nonconvex_global}
\caption{The ADMM for finding the distance between the boundaries of ellipsoids with a heuristic restart.}

\noindent\textbf{Step~1.} {Choose $\varepsilon_0 > 0$ and apply Algorithm~\ref{alg:ADMM_nonconvex} with an arbitrary
starting point. The algorithm terminates at a point $x^*$.}

\noindent\textbf{Step~2.} {If $\| x^*_1 - x^*_2 \| < \varepsilon_0$, then \textbf{return} $x^*$. Otherwise, choose
$\lambda_i^0 \in \mathbb{R}^d$, $i \in \{ 1, 2 \}$, and apply Algorithm~\ref{alg:ADMM_nonconvex} with 
$y_i^0 = - S_i x_i^* + c_i$, and the rest of the parameters defined as during the first run. The algorithm terminates at
a point $x^{**}$.}

\noindent\textbf{Step~3.} {If $\| x^*_1 - x^*_2 \| \le \| x^{**}_1 - x^{**}_2 \|$, \textbf{return} $x^*$. Otherwise,
\textbf{return} $x^{**}$.
}
\end{algorithm}

To verify the proposed heuristic restarting procedure we used a global optimisation method, which is a slight
modification of the global method for finding the so-called \textit{signed} distance between ellipsoids developed in
\cite{IwataNakatsukasaTakeda}. This method is based on the use of KKT optimality conditions for problem
\eqref{prob:EllipsoidDist_NonConvex} and a direct computation of all Lagrange multipliers for this problem via an
auxiliary generalised eigenvalue problem.

Recall that the problem of finding the distance between the boundaries of ellipsoids has the form:
\begin{align*}
  &\min\: \| x_1 - x_2 \|^2 \\ 
  & \text{subject to} \quad
  \langle x_1 - z_1, Q_1 (x_1 - z_1) \rangle = 1, \quad 
  \langle x_2 - z_2, Q_2 (x_2 - z_2) \rangle = 1.
\end{align*}
The KKT optimality conditions for this problem can be written as follows:
\begin{equation} \label{eq:KKT_Nontransformed}
  x_1 - x_2 = \mu Q_1(x_1 - z_1), \quad x_2 - x_1 = \gamma Q_2(x_2 - z_2).
\end{equation}
Here $\mu, \gamma \in \mathbb{R}$ are Lagrange multipliers. In \cite[Section~4.1]{IwataNakatsukasaTakeda} it was shown
that these Lagrange multipliers are solutions of the two following linear generalised eigenvalue problems:
\begin{align} \label{prob:GenEigenvalueProb1}
  \determ L_1(\mu) &:= \determ\Big[ \mu \big( F_{11} \otimes G_{10} - F_{10} \otimes G_{11} \big) 
  + \big( F_{01} \otimes G_{10} - F_{10} \otimes G_{01} \big) \Big] = 0, \\
  \determ L_2(\gamma) &:= \determ\Big[ \gamma \big( F_{11} \otimes G_{01} - F_{01} \otimes G_{11} \big) 
  + \big( F_{10} \otimes G_{01} - F_{01} \otimes G_{10} \big) \Big] = 0, \label{prob:GenEigenvalueProb2}
\end{align}
where $\otimes$ is the Kronecker product,
\begin{align*}
  F_{10} &= \begin{pmatrix} \mathbb{O}_d & - Q_2^{-1} \\ - Q_2^{-1} & \mathbb{O}_d \end{pmatrix}, \quad
  F_{01} = \begin{pmatrix} Q_1^{-1} & - Q_1^{-1} \\ - Q_1^{-1} & (z_1 - z_2)(z_1 - z_2)^T \end{pmatrix},
  \\
  G_{01} &= \begin{pmatrix} Q_2^{-1} & - Q_2^{-1} \\ - Q_2^{-1} & (z_1 - z_2)(z_1 - z_2)^T \end{pmatrix}, \quad
  G_{10} = \begin{pmatrix} \mathbb{O}_d & - Q_1^{-1} \\ - Q_1^{-1} & \mathbb{O}_d \end{pmatrix}, 
  \\
  F_{11} &= G_{11} = \begin{pmatrix} \mathbb{O}_d & I_d \\ I_d & \mathbb{O}_d \end{pmatrix},
\end{align*}
and, as above, $\mathbb{O}_d$ is the zero matrix of order $d$, while $I_d$ is the identity matrix of order $d$. 

Thus, one can globally solve the problem of finding the distance between the boundaries of ellipsoids by solving 
generalised eigenvalue problems \eqref{prob:GenEigenvalueProb1} and \eqref{prob:GenEigenvalueProb2}, computing
corresponding $x_1$ and $x_2$ for real values of $\mu$ and $\gamma$ from the KKT optimality conditions
\eqref{eq:KKT_Nontransformed}, and then comparing the distances $\| x_1 - x_2 \|$ for those $x_1$ and $x_2$ that satisfy
the constraints. The pair with the least distance is a globally optimal solution.

\begin{remark}
Let us note that by \cite[Theorem~4.4]{IwataNakatsukasaTakeda} for almost all positive definite matrices $Q_1$ and $Q_2$
generalised eigenvalue problems \eqref{prob:GenEigenvalueProb1} and \eqref{prob:GenEigenvalueProb2} have only a finite
number of solutions. However, these problems might be degenerate, if both $Q_1$ and $Q_2$ have eigenvectors orthogonal
to $z_1 - z_2$. In \cite[Section~5]{IwataNakatsukasaTakeda} regularity tests for problems
\eqref{prob:GenEigenvalueProb1} and \eqref{prob:GenEigenvalueProb2} and some techniques for handling degenerate cases
are discussed. However, we did not employ them in our numerical experiments, since we used the method from 
\cite{IwataNakatsukasaTakeda} only as a tool for verifying Algorithm~\ref{alg:ADMM_nonconvex_global}. In the case 
when the generalised eigenvalue problems were degenerate, we simply restarted computations with different matrices
$Q_1$ and $Q_2$. 
\end{remark}

\subsection{Numerical experiments}
\label{subsect:NumericalExperiments_NonconvexCase}

Let us present some results of preliminary numerical experiments demonstrating the higher efficiency of the ADMM in
comparison with other methods for finding the distance between the boundaries of ellipsoids in the case of nonconvex
high-dimensional problems. 

Our main goal was to test Algorithm~\ref{alg:ADMM_nonconvex_global} in the \textit{nonconvex} case. Therefore 
the problem data was generated in such a way that in most cases one ellipsoid lies within the other, which makes 
the problem multiextremal (i.e. there are locally optimal solutions of this problem that are not globally optimal). To
this end, we randomly generated a matrix $A$ of dimension $d \in \mathbb{N}$, whose elements were uniformly distributed
in the interval $[-100, 100]$. If the matrix $A$ has full rank, we define $Q_1 = A^T A$. Otherwise, the matrix $A$ was
randomly generated again till it had full rank, to ensure that the matrix $Q_1$ is positive definite. We chose as $Q_2$
a random diagonal matrix whose diagonal elements are uniformly distributed in the interval $[0.1, 0.6]$.
The centres $z_i$ of the ellipsoids were also randomly generated in such a way that their coordinates are uniformly
distributed in the interval $[-0.05, 0.05]$. Our numerical experiments demonstrated that this choice of matrices $Q_i$
and vectors $z_i$ ensured that in most cases ellipsoid $\mathcal{E}_1$ lies within $\mathcal{E}_2$. Moreover, the
numerical experiments also showed that when the dimension $d$ is small, the boundaries of the ellipsoids do not
intersect in almost all cases. However, as the dimension increases, the average distance between the boundaries of the
ellipsoids tends to zero and cases when the boundaries intersect appear more frequently. 

The parameters of Algorithms~\ref{alg:ADMM_nonconvex} and \ref{alg:ADMM_nonconvex_global} were chosen in the following
way. We set $\eta = 0.99$, $\beta = 2$, $\tau_0 = 10$, $\varepsilon = \varepsilon_0 = 10^{-6}$, and $\varkappa = 0.1$.
We also defined $\lambda^0 = 0$ and $y_1^0 = y_2^0 = (1, 0, 0, \ldots, 0)^T$ to make sure that in the case 
$z_1 = z_2 = 0$ the first iteration of the algorithm is not spent on correcting the initial data. Note that in the case
when $\lambda^0 = y^0 = 0$ and $z_1 = z_2 = 0$ one has $x^1 = 0$, $y^1$ is defined as an arbitrary vector from the unit
sphere, and $\lambda^1 = \tau_0 y^1$, which can be viewed as a random reinitialization of the algorithm. Let us point
out that in our numerical experiments any choice of $y^0_i$ from the unit sphere guaranteed that $v_i^n \ne 0$ for all
$n \in \mathbb{N}$ on Step~5 of Algorithm~\ref{alg:ADMM_nonconvex}.

We compared Algorithm~\ref{alg:ADMM_nonconvex_global} with the exact penalty method from \cite{TamasyanChumakov}, since
to the best of the author's knowledge this is the only available numerical method for finding the distance between the
boundaries of ellipsoids. The penalty parameter $\lambda$ for this method was defined as $\lambda = 100$. To make a
fair comparison we applied the same restarting procedure to the exact penalty method \cite{TamasyanChumakov}, as
described in Algorithm~\ref{alg:ADMM_nonconvex_global}. We used the inequality $\| G^*(z_k) \| < \varepsilon = 10^{-4}$
as a stopping criterion, since in some cases the algorithm failed to terminate before reaching the prespecified maximal
number of iterations when the value $\varepsilon = 10^{-5}$ was used. We chose an initial point for the method in the
same way as in the first version of the exact penalty method in the convex case (see
Section~\ref{subsect:NumericalExperiments_ConvexCase}). Namely, we defined $x_1^0 = z_1$ and 
$x_2^0 = z_2 + (0.1, \ldots, 0.1)^T$. Below we denote the exact penalty method as $EP$. 

Both algorithms were implemented in \textsc{Matlab}. We terminated the algorithms if the number of iterations exceeded
$10^6$. Similar to the convex case, we generated 10 problems for a given dimension 
$d \in \{ 5, 10, 20, 30, 50, 100, 200, 300,$ $500, 1000 \}$ and run each algorithm on these 10 problems. The total
run time of each method rounded to the nearest tenth is presented in Figure~\ref{fig:NonconvexCase} and
Table~\ref{tab:NonconvexCase}. In addition, we implemented a modification of the global method from
\cite{IwataNakatsukasaTakeda} described in Section~\ref{sect:HeuristicRestart} to verify whether the proposed heuristic
restarting procedure allows one to find a \textit{globally} optimal solution. Since the complexity of the global
method is very high (it is equal to $O(d^6)$; see \cite{IwataNakatsukasaTakeda}), we applied it only to problems of
dimensions $d \in \{ 5, 10, 20 \}$. For the sake of completeness, the total run time of the global method is given
in Table~\ref{tab:NonconvexCase}.

\begin{figure}[t]
\centering
\includegraphics[width=0.8\linewidth]{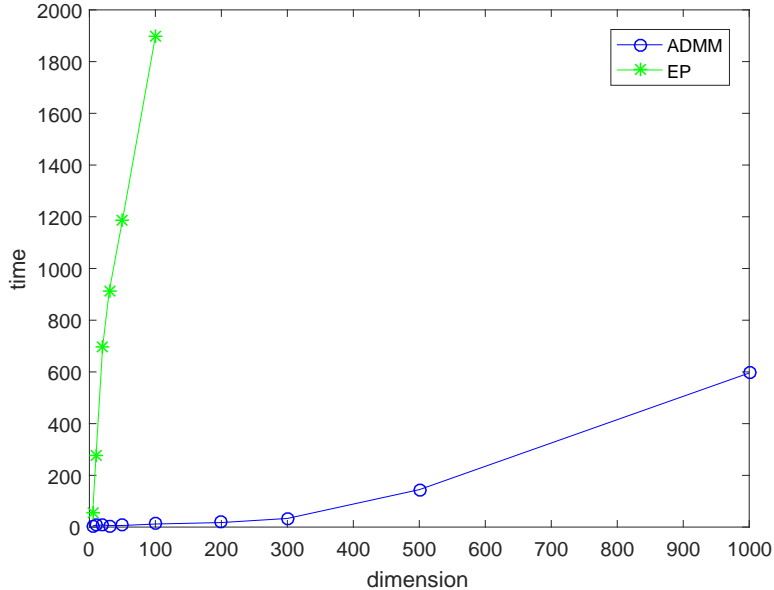}
\caption{The run time of the ADMM and the exact penalty method for 
$d \in \{ 5, 10, 20, 30, 50, 100, 200, 300, 500, 1000 \}$.}
\label{fig:NonconvexCase}
\end{figure}

\begin{table}
\begin{tabular}{c c c c c c c c c c c} 
  \hline
  d & 5 & 10  & 20 & 30 & 50 & 100 & 200 & 300 & 500 & 1000 \\
  \hline
  ADMM & 5.3 & 6.2 & 8.2 & 5.8 & 6.1 & 11.8 & 17.5 & 33.3 & 144.7 & 595.8 \\
  \hline
  EP & 55 & 274.1 & 695.7 & 911.2 & 1185.1 & 1896.8 & --- & --- & --- & --- \\
  \hline
  Global & 0.7 & 12.8 & 709.7 & --- & --- & --- & --- & --- & --- & --- \\ 
  \hline
\end{tabular}
\caption{The run time of each method in seconds. `---' indicates that either the time exceeded one hour or the algorithm
failed to find a solution due to reaching the prespecified maximal number of iterations.}
\label{tab:NonconvexCase}
\end{table}

Firstly, let us note that for \textit{all} problems of dimension $d \in \{ 5, 10, 20 \}$ (and numerous other test
problems not reported here) both Algorithm~\ref{alg:ADMM_nonconvex_global}, based on the ADMM, and the exact penalty
method from \cite{TamasyanChumakov} with the restarting procedure described in the previous section found
\textit{exactly the same} solutions as the global method. Thus, one can conclude that the proposed heuristic choice of a
new starting point for the second run of an algorithm indeed allows one to find a \textit{globally} optimal solution
for various problem instances. It seems that the reason behind this phenomenon lies in the fact that for any dimension
$d \in \mathbb{N}$ there are always only \textit{two} points of local minimum in the problem of finding the distance
between the boundaries of ellipsoids, while the rest of the KKT points are local/global maxima. Furthermore, the two
points of local minimum roughly lie on the `opposite sides' of the ellipsoids. Therefore, any reasonable minimisation
method would converge to one of the local minima, while restarting the method at the `opposite sides' of ellipsoids
allows one to compute the second point of local minimum and, as a result, find a global minimum between the two computed
points. A rigorous theoretical justification of this observation is a challenging open problem that lies beyond the
scope of this article. 

Let us also note that for problems of dimension $d \in \{ 30, 50, 100 \}$ Algorithm~\ref{alg:ADMM_nonconvex_global} and
the exact penalty method always found the same solution. In addition, we tested 
Algorithm~\ref{alg:ADMM_nonconvex_global} and the exact penalty method without the restarting procedure. The
results of our experiments showed that without the restart these methods often converge to a locally optimal solution,
which is not globally optimal. In particular, in the case $d = 5$ Algorithm~\ref{alg:ADMM_nonconvex_global} converged to
a locally optimal solution, which is not globally optimal, for $16$ out of $100$ test problems, while the exact
penalty method converged to such solutions for $42$ out of $100$ test problems. Thus, the proposed heuristic restarting
procedure is not redundant, since without it both the ADMM and the exact penalty method often cannot find a globally
optimal solution.

Secondly, the results of numerical experiments demonstrate that Algorithm~\ref{alg:ADMM_nonconvex_global} significantly
outperforms the exact penalty method from \cite{TamasyanChumakov} for all dimensions $d$. On the other hand, for small
dimensional problems ($d \le 7$) the global method was consistently faster than
Algorithm~\ref{alg:ADMM_nonconvex_global}. Thus, for small dimensional problems ($d \le 7$) it is reasonable to apply
the global method, while for problems of higher dimension Algorithm~\ref{alg:ADMM_nonconvex_global} is the method of
choice.

Finally, let us comment on the unexpected decrease of the run time of Algorithm~\ref{alg:ADMM_nonconvex_global}
for dimensions $d$ in the range between $20$ and $100$. It seems that this phenomenon is at least partially connected
with some peculiarities of our implementation of Algorithm~\ref{alg:ADMM_nonconvex_global} in \textsc{Matlab}. In
particular, we used standard \textsc{Matlab} routines to compute the Cholesky decomposition, to solve corresponding
systems of linear equations, etc. These routines might work more efficiently for problems of dimension more than $20$
due to the effect of parallelisation.

Moreover, it should be noted that for problems with $d \le 20$ reported here the boundaries of the ellipsoids did not
intersect. As our numerical experiments on various test problems showed, both Algorithm~\ref{alg:ADMM_nonconvex_global}
and the exact penalty method are extremely efficient, when it comes to finding an intersection point of the boundaries
of two ellipsoids. Even in the case $d = 1000$ it usually takes Algorithm~\ref{alg:ADMM_nonconvex_global} less than $5$
seconds to find an intersection point, provided the boundaries of the ellipsoids intersect. Therefore, the run time of
the algorithm significantly decreases with the increase of the number of problem instances when the boundaries
of the ellipsoids intersect. As was mentioned above, for our choice of the problem data the average distance between the
boundaries of the ellipsoids tends to zero as $d$ increases and the cases when the boundaries of the ellipsoids
intersect appear more frequently. This peculiarity of the problem data might be another factor that contributed to the
decrease in the run time. 

We tried generating matrices $Q_i$ in multiple different ways to ensure that the boundaries of the ellipsoids either do
not intersect for all $d$ or intersect with approximately the same frequency for all dimensions. However, generating
such data turned out to be a very complicated problem, whose detailed discussion lies outside the scope of this paper.

\section{Conclusions}

We developed several versions of the alternating direction method of multipliers for computing the distance between
either two ellipsoids or their boundaries. In the first case we presented the ADMM with both fixed penalty parameter and
automatic adjustments of the penalty parameter. In the second case we presented the ADMM with a heuristic rule for
updating the penalty parameter and a heuristic restarting procedure for finding a global minimum in the nonconvex case.
We numerically verified this procedure with the use of a slight modification of a global method for finding the
so-called signed distance between ellipsoids. The results of our numerical experiments showed that the proposed
heuristic restarting procedure always allows one to find a globally optimal solution. Furthermore, the results of
numerical experiments both in the convex and the nonconvex cases clearly demonstrate that the versions of the ADMM
developed in this paper significantly outperform all existing methods for finding the distance between two ellipsoids
and all methods for finding the distance between the boundaries of ellipsoids, except for the case of small-dimensional
problem.

%%%
%	Bibliography
%%%

\bibliographystyle{abbrv}  
\bibliography{Dolgopolik_bibl}

\end{document}